\documentclass[10pt]{amsart}
\usepackage{graphicx,epsfig}
\usepackage{bbding,amssymb,amsfonts,amsmath,amsthm,dsfont,wasysym,pifont,stmaryrd,inputenc}
\usepackage{epstopdf,yfonts}
\usepackage{hyperref,diagrams}

\usepackage{tikz,tikz-cd}
\usetikzlibrary{matrix}
\usetikzlibrary{matrix,shapes,backgrounds}

\newtheorem{theorem}{Theorem}[section]
\newtheorem*{theorem'}{Main Theorem}
\newtheorem*{theoremnn}{Theorem}

\newtheorem{prop}[theorem]{Proposition}

\newtheorem{facts}{Facts}

\newtheorem{cor}[theorem]{Corollary}
\newtheorem{remark}[theorem]{Remark}
\newtheorem{lemma}[theorem]{Lemma}
\newtheorem{definition}[theorem]{Definition}

\newarrow{Dashto}{}{dash}{}{dash}{>}
\renewcommand{\(}{\left(}
\renewcommand{\)}{\right)}

\renewcommand{\]}{\right]}

\renewcommand{\~}[1]{\overline{#1}}

\renewcommand{\geq}{\geqslant}
\renewcommand{\leq}{\leqslant}

\renewcommand{\>}{\right\rangle}
\newcommand{\8}{\infty}
\renewcommand{\.}{\cdot}
\renewcommand{\:}{\colon}
\renewcommand{\a}{\alpha}
\newcommand{\Aut}{\text{Aut}}

\newcommand{\ch}[1]{\check{#1}}

\renewcommand{\Cap}[2]{\underset{#1}{\overset{#2}{\cap} }}

\renewcommand{\Cup}[2]{\underset{#1}{\overset{#2}{\cup} }}

\newcommand{\f}{\varphi}

\newcommand{\frakH}{\mathfrak H}

\newcommand{\g}{\gamma}
\newcommand{\G}{\Gamma}
\newcommand{\Ga}{\mathcal{B}}

\newcommand{\I}{\mathcal{I}}

\renewcommand{\int}{\varint}
\newcommand{\Isom}{\mathrm{Isom}}

\newcommand{\M}{\mathcal{M}}

\renewcommand{\max}[1]{\underset{#1}{\mathrm{max}}}
\newcommand{\N}{\mathbb{N}}

\renewcommand{\o}{{{\scriptsize \sun}}}
\newcommand{\Om}{\Omega}
\newcommand{\om}{\omega}

\renewcommand{\P}{\mathcal{P}}

\newcommand{\R}{\mathbb{R}}
\renewcommand{\R}{\mathbb{R}}

\renewcommand{\S}{\mathcal{S}}

\newcommand{\stab}{\mathrm{stab}}

\newcommand{\Sum}[1]{\underset{#1}{{\sum} }}

\newcommand{\supp}{\mathrm{supp}}

\newcommand{\Z}{\mathbb{Z}}

  \def\firstcircle{(90:1cm) circle (1.5cm)}
  \def\secondcircle{(210:1cm) circle (1.5cm)}
  \def\thirdcircle{(330:1cm) circle (1.5cm)}

\title{The Furstenberg Poisson Boundary and CAT(0) Cube Complexes}
\author{Talia Fern\'os}
\thanks{The author was partially supported by NSF Grant number DMS-1312928, and UNCG New Faculty Summer Excellence Research Grant}
\address{Department of Mathematics and Statistics, University of North Carolina at Greensboro,  317 College Avenue, Greensboro, NC 27412, USA}
\email{\url{t_fernos@uncg.edu}}

\begin{document}
\maketitle

\begin{abstract}
We show under weak hypotheses that  $\partial X$, the Roller boundary of a finite dimensional CAT(0) cube complex $X$ is the Furstenberg-Poisson boundary of a sufficiently nice random walk on an acting group $\G$. In particular, we show that if $\G$ admits a nonelementary proper action on $X$, and $\mu$ is a generating probability measure of finite entropy and finite first logarithmic moment, then there is a $\mu$-stationary measure on $\partial X$ making it the Furstenberg-Poisson boundary for the $\mu$-random walk on $\G$. We also show that the support is contained in the closure of the regular points. Regular points exhibit strong contracting properties. 
\end{abstract}
{\footnotesize 
\tableofcontents
}
\normalsize
\section{Introduction}
CAT(0) cube complexes are fascinating objects of study, thanks in part to the interplay between two metrics that they naturally admit, the CAT(0) metric, and the median metric. Restricted to each cube, these  coincide either with the standard Euclidean metric ($\ell^2$) or with the ``taxi-cab'' metric ($\ell^1$).  Somewhat recently, CAT(0) cube complexes played a crucial roll in Agol's proof of the  Virtual Haaken Conjecture (an outstanding problem in the theory of 3-manifolds) \cite{Agol}, \cite{KahnMarcovic}, \cite{HaglundWise}, \cite{WiseMalnormal}, \cite{BergeronWise}. Examples of CAT(0) cube complexes and groups acting nicely on them include trees, (universal covers of) Salvetti complexes associated to right angeled Artin groups, Coxeter Groups, Small Cancellation groups, and are closed under taking finite products. 

Associated to a random walk on a group one has the Furstenberg-Poisson boundary. It is in some sense, the limits of the trajectories of the random walk.  Its existence, as an abstract measure space, for a generating random walk is guaranteed by the seminal result of Furstenberg \cite{Furstenberg}. This important object has since established itself as an integral part in the study of rigidity (see for example \cite{BaderFurman}
) in particular by realizing it as a geometric boundary of the group in question. 

One may associate to any CAT(0) space a visual boundary where each point is an equivalence class of geodesic rays.  The visual boundary for a CAT(0) space gives a compactification  of the space, at least when the space is locally compact \cite{Bridson_Haefliger}. For a wide class of hyperbolic groups, and more generally, certain groups acting on CAT(0) spaces, the visual boundary is a Furstenberg-Poisson boundary for suitably chosen random walks \cite{Kaimanovich1994}, \cite{KarlssonMargulis}.

The wall metric naturally leads to the Roller compactification of a CAT(0) cube complex. Nevo and Sageev show that the Roller boundary (see Section \ref{Subsection Roller Duality}) can be made to be a Furstenberg-Poisson boundary for a group $\G$ when the group admits a nonelementary proper co-compact action on $X$ \cite{NevoSageev}.  The purpose of this paper is to give a generalization of this result to groups which admit a nonelementary proper action on a finite dimensional CAT(0) cube complex. The complex is not assumed to be locally compact, and in particular, the action is not required to be co-compact. Our approach will be somewhat different to that of Nevo and Sageev and in particular, we will not address several of the dynamical questions that they consider: for example that the resulting stationary measure is unique, or that the  action is minimal or strongly proximal. Such questions will be examined in a forthcoming paper by L\'ecureux, Math\'eus, and the present author. 

Let $\mu$ be a probability measure on a discrete countable group $\G$. Assume  that it is generating, i.e. that the semi-group generated by the support of $\mu$ is the whole of $\G$. Recall that a probability measure $\mu$ on $\G$ is said to have finite entropy if  
$$H(\mu):= -\Sum{\g\in \G}{}\mu(\g)\log\mu(\g)<\8.$$ 
Also, if $|\cdot|: \G \to \R$ is a pseudonorm on $\G$ then $\mu$ is said to have finite first logarithmic moment (with respect to $|\cdot|$) if $\Sum{\g\in \G}{}\mu(\g)\log|\g|<\8$. (See Section \ref{Kaimanovich Strip} for more details.) If we have an action of $\G$ on $X$, then fixing a basepoint  $ o\in X$ allows us to consider the pseudonorm defined by $|\g|_o:=d(\g  o,  o)$.

\break

\begin{theorem'}\label{MainTheorem}
 Let $X$ be a finite dimensional CAT(0) cube complex, $\G$  a discrete countable group, $\G\to \Aut(X)$ a nonelementary proper action by automorphisms on  $X$, and $\mu$ a generating probability measure on $\G$ of finite entropy. If there is a base point $o\in X$ for which $\mu$ has finite first logarithmic moment then there exists a probability measure $\vartheta$ on the Roller boundary $\partial X$ such that  $(\partial X, \vartheta)$ is the Furstenberg-Poisson boundary for the $\mu$-random walk on $\G$. Furthermore, $\vartheta$ gives full measure to the regular points in $\partial X$. 
\end{theorem'}

The proof of the Main Theorem follows a standard path. We first show that the Roller Boundary is a quotient of the Furstenberg-Poisson boundary (Section 7) and then apply Kaimanovich's celebrated Strip Condition to prove maximality (Section 8). 

We note that Karlsson and Margulis show that the visual boundary of a CAT(0) space is the Furstenberg-Poisson boundary for suitable random walks \cite{KarlssonMargulis}. They assume very little about the space, but assume that the measure $\mu$ has finite first moment and that orbits grow at most exponentially. The Main Theorem above applies to the restricted class spaces (i.e.  CAT(0) cube complexes), which pays off by allowing for significantly weaker hypotheses on the action and the measure $\mu$. 

Observe that our Main Theorem applies for example to any non-elementary subgroup of a right angeled artin group or more generally of a graph product of finitely generated abelian groups \cite{RuaneWitzel}. 

Furthermore, we remark on the importance that the regular points are of full measure: they exhibit strong contracting properties. This will be exploited to study random walks on CAT(0) cube complexes in the forthcoming paper of L\'ecureux, Math\'eus, and the present author mentioned above.

An action on a CAT(0) cube complex is said to be Roller \emph{nonelementary} if every orbit in the Roller compactification is infinite (see Section \ref{Subsection Roller Duality}). This notion guarantees nonamenability of the closure of the acting group in $\Aut(X)$, and characterizes it for $X$ locally compact. This Tits' alternative, is essentially an encapsulation of results of Caprace and Sageev \cite{CapraceSageev}, Caprace \cite{CFI}, and Chatterji, Iozzi, and the author \cite{CFI}. It also comes after several versions of Tits' alternatives (see \cite{CapraceSageev}, \cite{Sageev_Wise_tits}). The statement is in the spirit of  Pays and Valette \cite{PaysValette}:

\begin{theorem}[Tits' Alternative]\label{Amen Closure}  Let $X$ be a  finite dimensional CAT(0) cube complex and $\G \leq \Aut(X)$.  If $X$ is locally compact then the following are equivalent:
\begin{enumerate}
\item \label{Interval} $\G$ does not preserve any interval $\I \subset \~X$.
\item\label{nonelem} The $\G$-action is  Roller nonelementary. 
\item\label{free subgroup}  $\G$ contains a  nonabelian  free subgroup acting freely on $X$. 
\item\label{nonamen closure} The closure $\~\G$ in $\Aut(X)$ is nonamenable. 
\end{enumerate}
\end{theorem}

\begin{remark}
In fact, the condition that $X$ be locally compact is only necessary for the implication  (\ref{nonamen closure}) $\implies $ (\ref{Interval}). All the other implications, namely (\ref{Interval}) $\iff$ (\ref{nonelem}) $\iff$(\ref{free subgroup}) $\implies$ (\ref{nonamen closure}) hold in the non-locally compact case as well. 
\end{remark}

\emph{Acknowledgements:} The author is grateful to the following people for their kindness and generosity: Uri Bader, Greg Bell, Ruth Charney, Indira Chatterji, Alex Furman, Alessandra Iozzi, Vadim Kaimanovich, Jean L\'ecureux,  Seonhee Lim, Amos Nevo, Andrei Malyutin, Fr\'ed\'eric Math\'eus, and Michah Sageev. Conversations and collaborations with these people made this article possible. Further thanks go to the University of Illinois at Chicago, the Centre International de Rencontres Math\'ematique, and the Institut Henri Poincar\'e. 


\section{CAT(0) Cube Complexes and Medians}\label{subsec:2.1}

We will say that a metric space is a {\em Euclidean cube} if there is an $n\in \N$ for which it is isometric to $[0,1]^n$ with the standard induced Euclidean metric from $\R^n$.

\begin{definition}
A second countable finite dimensional simply-connected  metric polyhedral complex $X$ is a {\em CAT(0) cube complex} if the closed cells are Euclidean cubes, the gluing maps are isometries and
the link of each vertex is a flag complex.
\end{definition}

Recall that a {\em flag complex} is a simplicial complex in which each complete subgraph on $(k+1)$-vertices is the 1-skeleton of a $k$-simplex in the complex. That the link of every vertex is a flag complex is equivalent to the condition of being locally CAT(0), thanks to Gromov's Link Condition. 

We remark that we absorb the condition of finite dimensionality in the definition of a CAT(0) cube complex and as such, we will not explicitly mention it in the sequel. Furthermore, if the dimension of the CAT(0) cube complex is $D$, then this is equivalent to the existence of a maximal dimensional cube of dimension $D$.

A {\em  morphism} between two CAT(0) cube complexes is an
isometry that preserves the cubical structures, i.e. it is an isometry $f:X\to Y$ such that $f(C)$ is a cube of $Y$ whenever $C$ is a cube in $X$. We denote by
$\Aut(X)$ the group of automorphisms of $X$ to itself.

\subsection{Walled Spaces}\label{Walled Spaces}
A \emph{space with walls} or a \emph{walled space} is a set $S$ together with a countable collection of non-empty subsets $\frakH \subset 2^S$ called half-spaces with the following properties:

\begin{enumerate}
\item If $h\in \frakH$ then $h\neq \varnothing$.
\item There is a fixed-point free involution $*\: \frakH\to \frakH$

\begin{equation*}
h\mapsto\*h:=X\setminus h.
\end{equation*}
\item The collection of half-spaces separating two points of $S$ is finite, i.e. for every $p,q\in S$ the set of half-spaces $h\in \frakH$ such that $p\in h$ and $q\in h^*$ is finite. 
\item There is a $D\in \N$ such that for every collection of pairwise transverse half-spaces, $\{h_1, \dots, h_n\}$ we must have that $n\leq D$. 
\end{enumerate}

A pair of half-spaces $h,k \in \frakH$ is said to be \emph{transverse} if the following four intersections are all non-empty:
$$h\cap k, \; h\cap k^*, \; h^*\cap k^*, \; h^*\cap k.$$

Associated to a walled space is the wall pseudo-metric $d: S\times S \to \R$:
$$d(p,q) = \frac{1}{2}\#\(\{h\in \frakH: p\in h, q\in h^*\}\cup \{h\in \frakH: q\in h, p\in h^*\}\).$$
This satisfies the properties of a metric, with the exception that $d(p,q) =0$ does not necessarily imply that $p=q$.

Let us then consider the associated quotient $S_\sim$ consisting of equivalence classes of points of $S$ whose pseudo-wall distance is 0. Clearly, the wall pseudo-metric descends to a metric on $S_\sim$. 

For $h\in \frakH$ the  \emph{wall} associated to $h$ is the unordered pair $\{h, h^*\}$. This explains the terminology, as well as the factor of $\frac{1}{2}$ in the definition of the (pseudo-)wall metric.


\subsection{CAT(0) Cube Complexes as Walled Spaces}

As we shall now see, CAT(0) cube complexes naturally admit a walled (pseudo-)metric and are in some sense the unique examples of such spaces.

Let $[0,1]^n$ be an $n$-dimensional cube. The $i$th coordinate projection is denoted by $\mathrm{pr}_i : [0,1]^n \to [0,1]$. A \emph{wall} of a cube $[0,1]^n$ is the set $\mathrm{pr}_i ^{-1}\{1/2\}$. Observe that the complement of each wall in a cube has two connected components. 

\begin{definition}
 A \emph{wall} of a CAT(0) cube complex $X$ is a convex subset whose intersection with each cube is either a wall of the cube or empty.  
\end{definition}

The complement of a wall in a CAT(0) cube complex has two connected components \cite[Theorem 4.10]{Sageev_95} which we call half-spaces and we denote them by $\frakH(X)$. Observe that since $X$ is second countable, there are countably many half-spaces in $\frakH(X)$. 

The notation and terminology here is purposefully chosen to remind the reader of a walled space. Indeed, in essence, a walled space uniquely generates a CAT(0) cube complex \cite{Sageev_95}, \cite{Chatterji_Niblo}, \cite{Nica}. And it is this walled space structure of the CAT(0) cube complex that we will ultimately be interested in, if not fascinated by. Since walls separate points in the zero-skeleton of a CAT(0) cube complex, we will in fact consider the zero-skeleton as our object of study.

Let $X_0$ denote the vertex set of $X$ and $\frakH(X_0) = \{h\cap X_0: h\in \frakH(X)\}$. 
 This yields a fixed-point free involution $*\: \frakH(X_0) \to \frakH(X_0) $
\begin{equation}\label{eq:involution}
h_0\mapsto\*h_0:=X_0\setminus h_0.
\end{equation}

One drawback of passing to the zero-skeleton, is that a wall is no longer a subset of $X_0$. Therefore, for $h_0\in \frakH(X_0)$, we will denote by $\^h_0$ the pair $\{h_0,\*h_0\}$ and think of it as a wall, as in Section \ref{Walled Spaces}. 

\begin{theorem}[\cite{Sageev_95},\cite{Nica},\cite{Chatterji_Niblo}]\label{CCCMagic}
 Let $(S,\frakH)$ be a walled metric space. Then, there exists a CAT(0) cube complex $X$ and an embedding $\iota: S\hookrightarrow X_0$ such that:
 \begin{enumerate}
\item If $S$ and $X_0$ are endowed with their respective wall metrics then $\iota$ is an isometry onto its image.
\item The set map induced by $\iota$ is a bijection $\frakH \to\frakH(X_0)$, $h\mapsto k$ such that 
$$k\cap \iota(S) = \iota(h).$$
\item If $\g: S\to S$ is a wall-isometry then there exists a unique extension to an automorphism $\g_0: X_0\to X_0$ that  agrees with $\g$ on $\iota(S)$. 
\end{enumerate}
Furthermore, if $(X_0, \frakH(X_0))$ is the walled space associated to the vertex set of a CAT(0) cube complex $X$, then the above association  applied to $(S,\frakH) = (X_0, \frakH(X_0))$ yields once more $X$, and $\iota : X_0 \to X_0$ can be taken to be the identity and the induced homomorphism $\Aut(X_0) \to \Aut(X_0)$ is the identity isomorphism. 
\end{theorem}

When a collection of half-spaces $\frakH$ is given, we will denote the associated CAT(0) cube complex as $X(\frakH)$,  leading to the somewhat abusive formulation of the last part of Theorem \ref{CCCMagic}: 
$$X(\frakH(X_0)) = X.$$

This striking result shows that the combinatorial information of the wall structure completely captures the geometry of the CAT(0) cube complex. This will be exploited in what follows. To this end, we now set $X= X_0$, and $\frakH=\frakH(X_0)$. Unless otherwise stated, every metric property will be taken with respect to the wall metric.

The first of many beautiful properties of CAT(0) cube complexes is  a type of Helly's Theorem:

\begin{theorem}\cite{Roller}\label{Helly}
 Let $h_1, \dots, h_n \in \frakH$ be half-spaces. If $h_i\cap h_j\neq \varnothing$ then 
 $$\Cap{i=1}{n} h_i \neq \varnothing.$$
\end{theorem}

 Keeping with the terminology of transverse half-spaces introduced in Section \ref{Walled Spaces},  if $h_1, \dots, h_n\in \frakH$ are pairwise transverse half-spaces then  $n\leq D$.

\subsection{Roller Duality}\label{Subsection Roller Duality}

Given a subset $\mathfrak{s}\subset\frakH$ of halfspaces, 
we denote by $ \mathfrak{s}^* $ the collection $\{\*h:\,h\in \mathfrak{s}\}$.
We say that $ \mathfrak{s} $ satisfies:

\begin{enumerate}
\item the {\em totality} condition if $ \mathfrak{s} \cup \mathfrak{s}^* =\frakH$;
\item the {\em consistency} condition if, $\mathfrak{s}\cap \mathfrak{s}^* = \varnothing$ and if $h\in \mathfrak{s}$ and $h\subset k$,
then $k\in \mathfrak{s} $.
\end{enumerate}

Fix $v\in X$ and consider the collection $U_v= \{h\in\frakH: v\in h\}$. It is straightforward to verify that $U_v$ satisfies both totality and consistency as a collection of half-spaces. 
 \emph{Roller Duality} is then obtained via the following observation: 
$$\Cap{h \in U_v}{} h = \{v\}.$$
This shows that if $w\in X$ then we have that 
$$U_v= U_w \iff v=w,$$
giving an embedding $X\hookrightarrow 2^\frakH$ obtained by $v\mapsto U_v$. This embedding is made isometric by endowing  $2^\frakH$  with the extended metric\footnote{ We note that this extended metric is not continuous unless $\frakH$ is finite.}
$$d(A,B) = \frac{1}{2}\#(A\triangle B).$$

For now, let us consider $X\subset 2^\frakH$. Then, the \emph{Roller compactification} is denoted by $\~X$ and is the closure of $X$ in $2^\mathfrak H$. The \emph{Roller boundary} is then $\partial X = \~X \setminus X$. Observe that in general, while $\~X$ is a compact space containing $X$ as a dense subset, it is not a compactification in the usual sense. Indeed, unless $X$ is locally compact, the embedding $X\hookrightarrow \~X$ does not have an open image, and $\partial X$ is not closed. This is best exemplified by taking the wedge sum of countably many lines. The limit of any sequence of distinct points in the boundary will be the wedge point. While it is also true that the visual boundary is not a compactification when $X$ is not locally comapact, the Roller boundary does present one significant advantage: the union $X\sqcup \partial X$ is indeed compact.

With this notation in place, the partition $\{h, h^*\}$  extends to a partition of $\~X$ and hence, when we speak of a half-space as a collection of points, we mean
$$h \subset \~X = h\sqcup h^*.$$ 

\begin{remark}\label{What is a wall}
Given $h\in \frakH$, we denote the set $\{h,h^*\}$ by $\^h$. By abuse of notation, for $k\in \frakH$, we will say that $\^h\subset k$ if and only if $h\subsetneq k$ or $h^*\subsetneq k$. This is consistent with the standard notion of the wall corresponding to a \emph{mid-cube}.
\end{remark}

We now give characterizations of special types of subsets of $\~X$. To this end, we say that $\mathfrak{s}\in 2^\frakH$ satisfies the \emph{descending chain condition} if every infinite descending chain of half-spaces is eventually constant.

\break

\begin{facts} The following are true for a non-empty $\mathfrak{s}\in 2^\frakH$:
\begin{enumerate}
\item If $\mathfrak{s}$  satisfies the consistency condition then
$$\varnothing \neq \Cap{h\in \mathfrak{s}}{} h \subset \~X.$$
\item \label{DCC} If $\mathfrak{s}$ satisfies the consistency condition and the descending chain condition then 
$$\varnothing \neq \(\Cap{h\in \mathfrak{s}}{} h \)\cap X.$$
\item The collection $\mathfrak{s}$  satisfies both the totality and consistency conditions if and only if there exists $v\in \~X$ such that $\mathfrak{s}= U_v$. Fixing $U_v\in 2^\frakH$ we have that 
\begin{itemize}
\item $v\in X$ if and only if $U_v$  satisfies the descending chain condition.
\item $v\in \partial X$ if and only if $U_v$ contains a nontrivial infinite descending chain, i.e. for each $n$ there is an $h_n \in S$ such that $h_{n+1}\subsetneq h_n$.
\end{itemize}
\end{enumerate}
\end{facts}

Let us say a few words about why these facts are true, or where one can find proofs, though likely several proofs are available. In case of Item (1), this is simple if one can show that the collection has the finite intersection property as $\~X$ is compact. Furthermore, the CAT(0) cube complex version of Helly's Theorem \ref{Helly} allows one to pass from finite intersections to pairwise intersections, and this last case is easy to verify given the condition of consistency.
For the second item, we refer the reader to Lemma 2.3 of \cite{NevoSageev}. 
Finally, for the last item, we refer the reader to \cite{Roller}.

There are also other special sets which will be of interest:

\begin{definition}\label{Nonterminating} 
The collection of \emph{nonterminating} elements is denoted by $\partial_{NT}X$ and consists of the elements $v\in \partial X$ such that  every finite descending chain can be extended, 
i.e. given $h\in  U_v$ there is a $k\in U_v$ such that 
$$k\subset h.$$ 
\end{definition}

In general, it may be the case that $\partial_{NT}X$ is empty. However, in case $X$ admits a nonelementary action  (see Section \ref{Section Nonelem})  then  $\partial_{NT}X$ is not empty \cite{NevoSageev},  \cite{CFI}.

\subsection{The Median}\label{subsec:intervals-median}

The vertex set of a CAT(0) cube complex 
with the edge metric (equivalently with the wall metric) is a median space \cite{Roller}, \cite{Chatterji_Niblo}, \cite{Nica}. The median structure extends nicely to the Roller compactification. 

We define the interval:
\begin{equation*}
\I(v,w) := \{m\in \~X :  U_v\cap U_w \subset  U_m\}.
\end{equation*}
In the special case that $v,w\in X$ this is the collection of vertices that are crossed by an edge geodesic connecting $v$ and $w$.

Then, the fact that $\~X$ is a \emph{median space}%
\footnote{ A median space is usually required to satisfy the condition that intervals are finite. However, we weaken this assumption here in order to extend the notion to the Roller compactification.
}
 is captured by the following: for every $u,v,w \in \~X$ there is a unique $m\in \~X$ such that 
$$\{m\}= \I(u,v) \cap \I(v,w) \cap \I(w,u).$$
This unique point is called the {\em median} of $u$, $v$, 
and $w$ and will sometimes be denoted by $m(u,v,w)$. In terms of half-spaces, we have:
$$\label{eq:median}
U_m = (U_u\cap  U_v) \cup ( U_v\cap  U_w)\cup ( U_w\cap  U_u),
$$
which is captured by this beautiful Venn diagram:
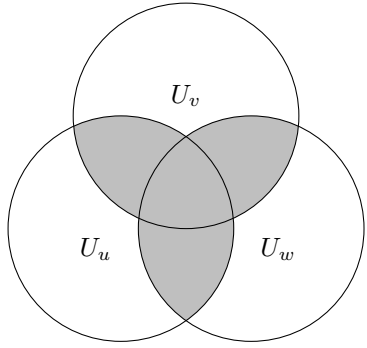
\begin{center}
 
    \begin{tikzpicture}
      \begin{scope}
    \clip \secondcircle;
    \fill[lightgray] \thirdcircle;
      \end{scope}
      \begin{scope}
    \clip \firstcircle;
    \fill[lightgray] \thirdcircle;
      \end{scope}
        \begin{scope}
    \clip \firstcircle;
    \fill[lightgray] \secondcircle;
      \end{scope}
      \draw \firstcircle node[text=black,above] {$U_v$};
      \draw \secondcircle node [text=black,below left] {$U_u$};
      \draw \thirdcircle node [text=black,below right] {$U_w$};
    \end{tikzpicture}
.
\end{center}

While general CAT(0) cube complexes can be quite wild,\footnote{  Indeed, if $T_\8$ is the tree of countably infinite valency, then the stabilizer group $\stab(v)$ of any vertex $v\in T_\8$ contains every discrete countable group.} the structure of intervals is  tamable by the following (see
\cite[Theorem 1.16]{BCGNW}):

\begin{theorem}\label{embedintervals}\label{th:Euclidean equivalences}\cite[Theorem 1.16]{BCGNW} Let $v,w\in\~{X}$.  
Then the vertex interval $\I(v,w)$ isometrically embeds into $\~{\Z}^D$ 
(with the standard cubulation) where $D$ is the dimension of $X$.
\end{theorem}

The proof of this employs Dilworth's Theorem which states that a partially ordered set has finite width $D$ if and only if it can be partitioned into $D$-chains. Here the partially ordered set is $U_w\setminus U_v$. Set inclusion yields the partial order and an antichain corresponds to a set of pairwise transverse half-spaces. By reversing the chains of half-spaces in $U_w\setminus U_v$ in a consistent way, we may find other pairs $x,y \in \~X$ such that $\I(x,y) = \I(v,w)$. This yields the following:

\begin{cor}\label{Interval on Finitely Many}
If $X$ has dimension $D$, then for any interval $I\subset \~X$, there are at most $2^D$ elements on which $I$ is an interval.
 \end{cor}

\subsection{Projections and Lifting Decompositions}\label{subsec:isom-emb}

It is straightforward, thanks to Theorem \ref{CCCMagic} that if $\frakH' \subset \frakH$ is an involution invariant subset, then there is a natural quotient map $X(\frakH) \to X(\frakH')$. Furthermore, if $\frakH'$ is $\G$-invariant for some acting group $\G$ then the quotient is $\G$-equivariant as well.   One can ask to what extent this can be reversed. Namely, when is it possible to find an embedding $X(\frakH') \hookrightarrow X(\frakH)$? And if $\frakH'$ is assumed to be $\G$-invariant, can the embedding be made to be $\G$-equivariant?

\begin{definition}\label{defi:lifting}  Given a subset $\frakH'\subset \frakH(X)$, a  {\em lifting decomposition} 
is a choice of a consistent subset $\mathfrak{s}\subset\frakH(X)$ such that
 $$\frakH(X)=\frakH'\sqcup(\mathfrak{s}\sqcup \mathfrak{s}^*).$$
\end{definition}

We note that a necessary condition for the existence of a lifting decomposition is that $\frakH'$ be involution invariant and that it be convex, i.e. if $h,k \in\frakH'$, and $h \subset \ell \subset k$ then $\ell \in \frakH'$.

Given a consistent set $\mathfrak{s}\subset \frakH(X)$, one can associate a set of walls (viewed as an involution invariant set of half-spaces)  $\frakH_\mathfrak{s} := \frakH(X) \setminus(\mathfrak{s}\sqcup \mathfrak{s}^*)$ so that $\mathfrak{s}$ is \emph{a} lifting decomposition of $\frakH_\mathfrak{s}$, though there could of course be others.

The terminology is justified by:

\break

\begin{prop}\label{LiftingDecomp} \cite[Lemma 2.6]{CFI} 
The following are true:\begin{itemize}
\item Suppose that $\frakH'\subset \frakH(X)$. If there exists $\mathfrak{s}$ a lifting decomposition for $\frakH'$ then there is an isometric embedding $\~X(\frakH') \hookrightarrow \~X$ induced from the map $2^{\frakH'} \hookrightarrow 2^{\frakH(X)}$ where $U \mapsto U\sqcup \mathfrak{s}$ and the image of this embedding is 
$$\Cap{h\in \mathfrak{s}}{} h\subset \~X.$$
\item Conversely, if $\mathfrak{s}\subset \frakH(X)$ is a consistent set of half-spaces, then, setting $\frakH_\mathfrak{s} = \frakH(X)\setminus(\mathfrak{s}\sqcup \mathfrak{s}^*)$ we get an isometric embedding 
$\~X(\frakH_\mathfrak{s}) \hookrightarrow \~X$ obtained as above, onto
$$\Cap{h\in \mathfrak{s}}{} h\subset \~X.$$
\item\label{DCClifting} If $\mathfrak{s}$ satisfies the descending chain condition, then the image of $X(\frakH_\mathfrak{s})$ is in $X$. 
\end{itemize}
Furthermore, if the set $\mathfrak{s}$ is $\G$-invariant then, with the restricted action on the image, the above natural embeddings are  $\G$-equivariant.
\end{prop}

\begin{remark}\label{rem half-spaces in proj}
 We note that the projection $X\to X(\frakH')$ obtained by forgetting the half-spaces $\frakH\setminus \frakH'$ is onto. This means that if there is a lifting decomposition $\~X(\frakH') \hookrightarrow \~X$ then the relationship between two half-spaces (i.e. facing, transverse, etc.) is equivalent if one considers them as half-spaces in $X$ or in $X(\frakH')$.
\end{remark}

Let us interpret the significance of Proposition \ref{LiftingDecomp} in the context of the collection of the involution-invariant set of half-spaces $\frakH(v,w):=U_v\triangle U_w$, for  $v,w\in \~X$. These are the half-spaces separating $v$ and $w$.  Then, the collection of half-spaces  $\frakH(v,w)^+ := U_v\cap U_w$, i.e. those that contain both $v$ and $w$, is a consistent set of half-spaces and it is straightforward to verify that $\frakH(v,w)^+$ is a lifting decomposition for $\frakH(v,w)$, yielding an isometric embedding of the  CAT(0) cube complex associated to $\frakH(v,w)$ onto $ \I(v,w)$.

\section{Three Key Notions}\label{subsec:essential}

There are three notions that together form a powerful framework within which to study CAT(0) cube complexes. The first is the classical notion of a nonelementary action. Caprace and Sageev showed that this allows one to study the \emph{essential core} of a CAT(0) cube complex \cite{CapraceSageev}, which is the second notion. Finally, Behrstock and Charney introduced the notion of strong separation which allows for the local detection of irreducibility \cite{Behrstock_Charney}, which was shown by Caprace and Sageev to be available in the nonelementary setting \cite{CapraceSageev}.

\subsection{Nonelementary Actions}\label{Section Nonelem}

As a CAT(0) space, a CAT(0) cube complex has  a visual boundary
$\partial_\sphericalangle X$ which is obtained by considering equivalence classes of geodesic rays, where two rays are equivalent if they are  at bounded distance
from each other. The topology on $\partial_\sphericalangle X$ is the cone topology (which coincides with the topology of uniform convergence on compact subsets, when one considers geodesic rays emanating from the same base point)
\cite{Bridson_Haefliger}. While the visual boundary is not well behaved for non-proper spaces in general, the assumption that the space is finite dimensional is sufficient \cite{CapraceLytchak}.

\begin{definition}
 An isometric action on a CAT(0) space is said to be elementary if there is a finite orbit in either the space or the visual boundary. 
\end{definition}

To exemplify the importance of this property, we have:

\begin{theorem}\cite{CapraceSageev}
Suppose $\G\to \Aut(X)$ is an action on the CAT(0) cube complex $X$. Then either the action is elementary, or $\G$ contains a freely acting non-abelian free group. 
 
\end{theorem}

\subsection{Essential Actions}

Caprace and Sageev \cite{CapraceSageev} showed that for nonelementary actions, there is a nonempty ``essential core" where  the action is well behaved. Let us now develop the necessary terminology and recall the key facts.

\begin{definition} Fix a group $\G$ acting by automorphisms on the CAT(0) cube complex $X$. A  half-space $h \in \frakH$ is called ...
\begin{itemize}
 \item  \emph{$\G$-shallow} if for some (and hence all) $x\in X$, the set $\G\cdot x \cap h$ is at bounded distance from $h^*$,  otherwise, it is said to be \emph{$\G$-deep}.
  \item  \emph{$\G$-trivial} if $h$ and $h^*$ are both shallow. 
    \item  \emph{$\G$-essential} if $h$ and $h^*$ are both deep. 
     \item  \emph{$\G$-half-essential} if it is deep and $h^*$ is shallow. 
\end{itemize}

\end{definition}

\begin{remark}\label{Finite Index Ess}
 Observe that the collections of essential and trivial half-spaces are both closed under involution and that  the collection of half-essential half-spaces is consistent. Furthermore, a half-space $h\in \frakH$ is $\G$-essential if and only if it is $\G_0$-essential for any $\G_0\leq \G$ of finite index. 
\end{remark}

\begin{theorem}\cite[Proposition~3.5]{CapraceSageev}\label{EssNonEmpty}
Assume $\G\to \Aut(X)$ is a nonelementary action on the CAT(0) cube complex $X$ then the collection of $\G$-essential half-spaces is non-empty. Furthermore, if $Y$ is the CAT(0) cube complex associated to the $\G$-essential half-spaces  then $Y$ is unbounded and there is a $\G$-equivariant embedding $Y \hookrightarrow X$.
\end{theorem}
 
 The image of $Y$ under this embedding is called the \emph{$\G$-essential core}. If all half-spaces are essential, then the action is said to be \emph{essential}. 
 
A simple but powerful concept introduced by Caprace and Sageev is that of flipping a half-space. A half-space $h\in \frakH$ is said to be $\G$-flippable if there is a $\g\in \G$ such that $h^* \subset \g h$. 

\begin{lemma} \cite[Flipping Lemma]{CapraceSageev}\label{flip}
Assume $\G\to \Aut(X)$ is  nonelementary. If $h\in \frakH$ is essential, then $h$ is $\G$-flippable. 
\end{lemma}

Recall that a measure $\lambda$ is said to be \emph{quasi-$\G$-invariant} whenever the following holds for every $\g \in \G$, and every measurable set $E$: if $\lambda(E) >0$ then $\lambda(\g E) >0$.

\begin{cor}\cite{CFI}\label{PosMeasHalfSpaces}
Suppose $\G\to \Aut(X)$ is a nonelementary and essential action  on the CAT(0) cube complex $X$. If $\lambda$ is a quasi-$\G$-invariant probability measure on $\~ X$ then $\lambda(h)>0$ for every half space $h\in \frakH(X)$. 
\end{cor}

\begin{proof}
Let $h\in \frakH$. Then $\lambda(h\sqcup h^*) = 1$ which means that either $\lambda(h) >0$ or $\lambda(h^*)>0$. If $\lambda(h^*)>0$ then apply the Flipping Lemma \ref{flip} and deduce that there is a $\g\in \G$ such that $h^* \subsetneq \g h$ and hence $\lambda(\g h) \geq \lambda(h^*) >0$. But of course, $\lambda$ is $\G$-quasi-invariant so $\lambda(h)>0$.
\end{proof}

Another very important operation on half-spaces developed by Caprace and Sageev is the notion of double skewering:
\begin{lemma}\cite[Double Skewering]{CapraceSageev}\label{Double Skewering Lemma}
Suppose $\G\to \Aut(X)$ is a nonelementary  action on the CAT(0) cube complex $X$. If $h\subsetneq k$ are two essential half spaces, then there exists a $\g \in \G$ such that 
$$\g k\subsetneq h \subsetneq k.$$
\end{lemma}
 
The following is almost a direct consequence of the definitions. The reader will find a more in depth formulation in \cite[Proposition 3.2]{CapraceSageev} 

\begin{lemma}\label{lem:NoFiniteColInvHalf}
 If $\G$ acts on the CAT(0) cube complex $X$ and preserves a finite collection of half-spaces then the $\G$-action is either elementary or not essential.
\end{lemma}

 An action of $\G$ on $X$ is said to be \emph{Roller nonelementary} if there is no finite orbit in the Roller compactification. Of course, having a finite orbit in $X$ is equivalent to having a fixed point, and so, what distinguishes the Roller nonelementary from the visual nonelementary actions is the existence of finite orbits in the corresponding boundaries. Furthermore, (visual) nonelementary actions are necessarily Roller nonelementary,  though the converse is false in general. One can take for example the standard action of $\G=\Z\times F_2$ on $\Z\times T$, where $T$ is the standard Cayley tree. It is straightforward to see that this example is essential and  elementary but not Roller elementary. On the other hand, if we set $\partial \Z =\{-\8, \8\}$ then both $\{\8\}\times T$ and $\{-\8\}\times T$ are $\G$-invariant and nonelementary. This phenomenon is captured in the following:

\begin{prop} (\cite[Proposition 2.26]{CFI}, \cite{CapraceSageev})\label{prop:visual-to-roller} 
Let $X$ be a finite dimensional CAT(0) cube complex
and let $\Gamma\to\Aut(X)$ be an action on $X$.  One of the following hold:
\begin{enumerate}
\item The $\G$-action is Roller-elementary. 
\item There is a finite index subgroup $\Gamma'<\Gamma$ and a $\Gamma'$-invariant 
subcomplex $\~X'\hookrightarrow \~ X$ associated to a $\G'$-invariant $\frakH'\subset \frakH(X)$ on which the $\Gamma'$-action is nonelementary and essential.
\end{enumerate}
Moreover, if the action of $\G$ is nonelementary on $X$ then  $X'\hookrightarrow X$ and $X'$ is the $\G$-essential core. 
\end{prop}

\subsection{Product Structures}\label{subsec:decomposition} 

A CAT(0) cube complex is said to be \emph{reducible} if it can be expressed as a nontrivial product. Otherwise, it is said to be \emph{irreducible}. A CAT(0) cube complex $X$ with half-spaces $\frakH$, admits a product decomposition $X=X_1\times \cdots \times X_n$ if and only if there is a decomposition 
$$\frakH= \frakH_1 \sqcup \cdots \sqcup \frakH_n$$
such that if $i\neq j$ then $h_i\pitchfork h_j$ for every $(h_i, h_j)\in \frakH_i\times \frakH_j$ and $X_i$ is the CAT(0) cube complex on half-spaces $\frakH_i$. 

\begin{remark}\label{Remark product intervals is interval in product}
 This means that an interval in the product is the product of the intervals. Namely if $(x_1, \dots, x_n), (y_1, \dots, y_n) \in \~X_1\times \dots \~X_n$ then $$\I\((x_1, \dots, x_n), (y_1, \dots, y_n)\) = \I(x_1, y_1) \times \cdots \times \I(x_n, y_n).$$
\end{remark}

The irreducible decomposition is unique (up to permutation of the factors) and $\Aut(X)$ contains $\Aut(X_1)\times \dots \times \Aut(X_n)$ as a finite index subgroup. Therefore, if $\G$ acts on $X$ by automorphisms, then there is a subgroup of finite index which preserves the product decomposition 
 \cite[Proposition 2.6]
{CapraceSageev}.

We take the opportunity to record here that the Roller boundary is incredibly well behaved when it comes to products:
\begin{equation*}
\partial X=\bigcup_{i=1}^{n}\~{X_1}\times\dots\~{X_{j-1}}\times\partial X_j\times
  \~{X_{j+1}}\times\dots\times\~{X_n}.
\end{equation*}

While the definition of (ir)reducibility for a CAT(0) cube complex in terms of its half-space structure is already quite useful, its global character makes it at times difficult to implement. Behrstock and Charney developed an incredibly useful notion for the Salvetti complexes associated to Right Angled Artin Groups, which was then extended by Caprace and Sageev.

\begin{definition}[\cite{Behrstock_Charney}]\label{defi:strongly separated} 
Two half-spaces $h, k \in \frakH$ are said to be \emph{strongly separated} if there is no half-space which is simultaneously transverse to both $h$ and $k$. For a subset $\frakH'\subset \frakH$ we will say that $h,k\in \frakH'$ are \emph{strongly separated in $\frakH'$} if there is no half-space in $\frakH'$ which is simultaneously transverse to both $h$ and $k$.
\end{definition}

The following is proved in \cite{Behrstock_Charney} for (the universal cover of) the Salvetti complex of non-abelian RAAGs:

\begin{theorem}\cite{CapraceSageev}\label{th:ss} Let $X$ be a finite dimensional irreducible CAT(0) cube complex such that the action of $\Aut(X)$ is essential and non-elementary. Then $X$ is irreducible if and only if there exists a pair of strongly separated half-spaces. 
\end{theorem}

\subsection{Euclidean Complexes}

\begin{definition} Let $X$ be a CAT(0) cube complex. 
We say that $X$ is {\em Euclidean} if the vertex set with the combinatorial metric 
embeds isometrically in $\Z^D$ with the $\ell^1$-metric, for some $D<\infty$.
\end{definition}

As our prime example of a Euclidean CAT(0) cube complex is an interval, which is the content of Theorem \ref{th:Euclidean equivalences}.

\begin{definition}
 An $n$-tuple of  half-spaces $h_1, \dots, h_n\in \frakH$ is said to be \emph{facing} if for each $i\neq j$ 
 $$h_i^* \cap h_j^* = \varnothing.$$
\end{definition}

As an obstruction to when a CAT(0) cube complex is Euclidean, there is the following:

\begin{lemma}\label{NtuplesNonEuclidean}\cite[Lemma 2.33]{CFI} 
If $X$ is a Euclidean CAT(0) cube complex
that isometrically embeds into $\Z^D$, then any set of pairwise facing halfspaces
has cardinality at most $2D$.
\end{lemma}

The following is an important characterization of when a complex is Euclidean.

\begin{cor}\label{finallyIntervalEuclidean}\cite[Corollary 2.33]{CFI}, \cite{CapraceSageev} 
 Let $X$ be an irreducible finite dimensional  CAT(0) 
cube complex so that the action of $\Aut(X)$ is essential and non-elementary. The following are equivalent:
\begin{enumerate}
\item $X$ is an interval;
\item $X$ is Euclidean;
\item $\frakH(X)$ does not contain a facing triple of half-spaces.
\end{enumerate}
\end{cor}

\begin{remark}
 The statement of \cite[Corollary 2.33]{CFI} 
  states that $X$ is Euclidean if and only if $\frakH(X)$ does not contain a facing triple. However, the proof actually shows that (2) implies (3) and (3) implies (1). The missing (1) implies (2) is of course provided by Theorem \ref{th:Euclidean equivalences} \cite{BCGNW}.
\end{remark}

\break

\begin{lemma}
 Let $X$ be an interval on $v,w\in \~X$. Then $\Aut(X)$ is elementary. 
\end{lemma}

\begin{proof}
If $X$ is an interval then  the collection of points on which it is an interval is finite and bounded above by $2^D$ by Corollary \ref{Interval on Finitely Many}. Let $\G_0$ be the finite index subgroup which fixes this set pointwise and let $v$ belong to this set. Then, for every finite collection $h_1, \dots, h_n \in U_v$ the intersection $\Cap{i= 1}{n} h_i$ is not empty. Hence, the intersection of the visual boundaries corresponding to the $h_i$ must be non-empty and its unique circumcenter is fixed for the $\G_0$-action by \cite[Proposition 3.6]{CapraceSageev}.
\end{proof}

\begin{lemma}\label{lem:boundary of factors} \cite[Lemma 2.28]{CFI}. 
Let  $\Gamma\to\Aut(X)$ be a non-elementary action. 
Then the $\Gamma_0$-action on the irreducible factors of the essential core is
also non-elementary and essential, where $\G_0$ is the finite index subgroup preserving this decomposition.
\end{lemma}

We immediately deduce (see also \cite[Corollary 2.34]{CFI}): 

\begin{cor}\label{Cor.Not an interval}
 If $\G\to \Aut(X)$ is nonelementary  then any irreducible factor in the essential core of $X$ is not Euclidean and hence not an interval.
\end{cor}

\subsection{The Combinatorial Bridge}
\label{bridge}

Behrstock and Charney showed that the CAT(0) bridge connecting two strongly separated walls is a finite geodesic segment \cite{Behrstock_Charney}. In \cite{CFI} this idea is translated to the ``combinatorial", i.e. median setting.  Most of what follows is from or adapted from \cite{CFI}, though the notation differs slightly. Recall our convention that,  $\{k,k^*\}$ is denoted by $\^k$, for a half-space $k$, and that given another half-space $h$ we will say that $\^k\subset h$ if either $k$ or $k^*$ is a \emph{proper} subset of $h$. 

\begin{remark}\label{shortcut}
 Observe that for two half-spaces $h,k$ we have that $h\cap k$ and $h^*\cap k$ are both nonempty if and only if $h\pitchfork k$ or $\^h \subset k$.
\end{remark}
 
\begin{definition}
Let $h_1\subsetneq h_2^*$ be a nested pair of halfspaces and $\beta(h_1, h_2)$ denote the collection of half-spaces $h\in \frakH$ such that one of the following conditions hold:
\begin{enumerate}
\item $\^h_1 \subset h$ and $h_2\pitchfork h$;
\item $\^h_2 \subset h$ and $h_1\pitchfork h$;
\item $\^h_1, \^h_2\subset h$.
\end{enumerate}
Furthermore, a half-space $h$ will be said to be of type (1), (2), or (3) if it satisfies the corresponding property. 
\end{definition}

We note that both $h_1$ and $h_2$ are not of types (1)--(3). Furthermore, since $h_1$ and $h_2$ are disjoint, condition (1)  actually means that $h_1\subset h$ and $h_2\pitchfork h$ (and analogously for condition (2)). 

\begin{lemma}\label{Bridge Consistent}
 Given $h_1\subsetneq h_2^*$, the collection $\beta(h_1, h_2)$ is consistent. Furthermore, $\beta(h_1, h_2)$ satisfies the descending chain condition. 
\end{lemma}

\begin{proof}
We begin by observing that if $h\in \beta(h_1, h_2)$ then we necessarily have that $h^*\notin\beta(h_1,h_2)$. 

Now suppose that $h\in \beta(h_1, h_2)$ is of type (3). If $h\subset k$ then clearly $k$ is also of type (3) and hence $k\in \beta(h_1, h_2)$. 

Next suppose that $h$ is of type (1) and $h\subset k$. Then, $\^h_1\subset k$. Since $h_2\pitchfork h$ and $h\subset k$ we have that $ h_2\cap k$ and $h_2^*\cap k $ are both nonempty. By Remark \ref{shortcut}, either $k\pitchfork h_2$ or $\^h_2\subset k$, and so $k\in \beta(h_1, h_2)$. 

Of course, a symmetric argument shows that if $h$ of type (2) and $h\subset k$ then $k\in \beta(h_1, h_2)$.

Next we turn to the question of the descending chain condition. Since there are finitely many half-spaces in between any two, an infinite descending chain will eventually fail to satisfy all three conditions (1) through (3).
\end{proof}

\begin{definition}
 The (combinatorial) \emph{bridge} between $h_1\subsetneq h_2^*$ is denoted by $B(h_1, h_2)$ and corresponds to $\Cap{h\in \beta(h_1, h_2)}{} h \subset X$.  
\end{definition}

\begin{lemma}\label{this lemma}
Assume that $h_1 \subsetneq h_2^*$ and set $\beta= \beta(h_1, h_2)$. The collection $\frakH'= \frakH\setminus(\beta\sqcup \beta^*)$ consists of half-spaces $h$ such that one of the following hold:
 \begin{itemize}
\item $h\pitchfork h_1$ and $h\pitchfork h_2$; 
\item up to replacing $h$ by $h^*$ we have that 
$$h_1 \subseteq h \subseteq h_2^*.$$ 
\end{itemize}
\end{lemma}

\begin{proof}
It is clear that if $h$ is a half-space that is transverse to both $h_1$ and $h_2$ then $h\notin\beta\sqcup \beta^*$. 
It is also clear that  if $\^h= \^h_1$ or $\^h= \^h_2$ then  $h\notin \beta\sqcup \beta^*$.

Now assume that $h_1 \subsetneq h \subsetneq h_2^*$.   Then   $h\supset\^h_1$ and $h$  does not contain nor  is it  transverse to $\^h_2$ and hence $h\notin\beta\sqcup \beta^*$.

Conversely, suppose that $h\notin\beta\sqcup \beta^*$. If $h\pitchfork h_1$ then since $h$ and $h^*$ are not type $(2)$ we must have that  $h\pitchfork h_{2}$. Assume then that $h$ is not transverse to both $h_1$ and $h_2$. 

If either $\^h = \^h_1$ or $\^h = \^h_2$ then up to replacing $h$ by $h^*$ we have that $h_1\subseteq h\subseteq h_2^*$. Therefore, suppose that $\^h \neq \^h_1$ and $\^h \neq \^h_2$. Then, each of $\^h_1$ and $\^h_2$ is contained in $h$ or $h^*$. Since both $h$ and $h^*$ are  not  of type (3), we must have that, up to replacing $h$ by $h^*$, $\^h_1\subset h$ and $\^h_2\subset h^*$. Now, of course, since $h_1\cap h_2=\varnothing$, we conclude that $h_1\subset h$ and $h_2\subset h^*$, i.e.
$$h_1\subset h\subset h_2^*.$$
\end{proof}

\begin{cor}\label{Cor unique endpoints bridge}
Assume $h_1 \subset h_2^*$ are strongly separated. With the notation as in Lemma \ref{this lemma}, $\beta$ gives a lifting decomposition of $\frakH'$. Furthermore,  there exists a unique $x_i \in h_i$ such that 
$$B(h_1, h_2) = \I(x_1, x_2).$$
\end{cor}

\begin{proof}
The fact that $\beta$ is a lifting decomposition for 
$$\frakH' = \{h\in \frakH: h_1\subseteq h\subseteq h_2^* \text{ or }h_1\subseteq h^*\subseteq h_2^* \}$$
follows from Lemmas \ref{Bridge Consistent} and \ref{this lemma}. In particular, $\frakH'$ is precisely the set of half-spaces which separate points in $B(h_1, h_2)$. 

Let us show that $B(h_1, h_2)$ is an interval. To this end, let $S_i = h_i \cap B(h_1, h_2)$. Since $h_i \in \frakH'$ it follows that $S_i\neq \varnothing$. 

Fixing $i$, suppose that $x,y\in S_i$. Then, any wall separating them must belong to $\frakH'$. By Lemma \ref{this lemma} and the assumption that $h_1$ and $h_2$ are strongly separated,  (again replacing $h$ by $h^*$ if necessary) we see that $h_1\subset h\subset h_2^*$. This means of course that $h_1\cap h^*=\varnothing$ and hence $x=y$, i.e. $S_i$ is a singleton, for both $i=1,2$.

Set $S_i = \{x_i\}$. Once more, since $h_1$ and $h_2$ are strongly separated, the collection $\frakH'$ corresponds to half-spaces nested in between $h_1$ and $h_2^*$ and hence  $\frakH'\subset \frakH(x_1, x_2)$. Conversely, if $h\in \frakH(x_1, x_2)$ then $h$ separates the two points $x_1, x_2\in B(h_1, h_2)$ and hence $h\in \frakH'$. 
\end{proof}

\begin{lemma}\label{ss Points are median}
 Assume that  $h_1\subset h_2^*$ are strongly separated. If $\xi_i\in h_i \subset \~X$, and $p\in B(h_1, h_2)$, then 
 $$m(\xi_1, p, \xi_2) = p.$$
\end{lemma}

\begin{proof}
 Let $m = m(\xi_1, p, \xi_2)$. Recall that $m$ is uniquely determined by 
 $$U_m = (U_{\xi_1}\cap U_{p})\cup (U_{p}\cap U_{\xi_2})\cup (U_{\xi_2}\cap U_{\xi_1}),$$ 
 and so we must show that if $h\in U_{\xi_2}\cap U_{\xi_1}$ then $h\in U_p$. In fact, we will show that if $\xi_1,\xi_2 \in h$ then $h \in \beta(h_1,h_2)\subset U_p$. 
 
By assumption $\xi_i \in h\cap h_i\neq \varnothing$. Furthermore, since $h_1$ and $h_2 $ are strongly separated, $h$ can not be transverse to both $h_1$ and $h_2$. Suppose that  $h$ is  parallel to $h_2$. Since, $\xi_2 \in h_2\cap h$ and $\xi_1 \in h_2^*\cap h$ by Remark \ref{shortcut} we have that $h\supset\^h_2$. The same argument shows that either $h$ is transverse to $h_1$ or contains $\^h_1$ and therefore  $h\in \beta(h_1, h_2)$. 
\end{proof}

\subsection{More Consequences}

\begin{lemma}\cite[Lemma 2.28]{CFI}\label{extend4ss} 
Suppose that $\G \to \Aut(X)$ is a nonelementary and essential action, with $X$ irreducible. 
\begin{itemize}
\item If $h\in \frakH$ then  there exists $\g, \g'\in \G$ such that the following are pairwise strongly  separated 
$$\g h\subset h\subset \g'h.$$ 
\item In each orbit, there are $n$-tuples of  facing and pairwise strongly separated half-spaces. 
\end{itemize}
\end{lemma}

\begin{lemma}\label{ssContaining x}
Let $X$ be an irreducible CAT(0) cube complex with a  nonelementary and essential $\G$-action.  Let $h\in \frakH$ and $n\in \Z$ with $n\geq 2$. Then, there exists an n-tuple  $\{k_1, \dots, k_n\}$ contained in a single $\G$-orbit that are  facing and pairwise strongly separated such that 
$$\^h \in \Cap{i = 1}{n} k_i.$$
\end{lemma}

\begin{proof}
 Fix $h\in \frakH$. For $n= 2$ we take $k_1 = \g h^*$ and $k_2= \g' h$ as in Lemma \ref{extend4ss}. 

Now, assume $n>2$. Let $\{b_1, \dots, b_{n+1}\}$ be the collection of  facing and pairwise strongly separated half-spaces guaranteed by Item (2) of Lemma \ref{extend4ss}.  For each $i = 1, \dots, n+1$ exactly one of the following possibilities hold:
\begin{itemize}
\item[(a)] $\^h = \^b_i$;
\item[(b)] $h\pitchfork b_i$; 
\item[(c)] $\^h\subset b_i^*$;  
\item[(d)] $\^h \subset b_i$. 
\end{itemize}

Furthermore, since the collection is strongly separated and facing, there is at most one $i$, assume it is $i = n+1$,  for which the mutually exclusive items (a) through (c) can occur. Therefore, we have that
$$\^h \subset \Cap{i = 1}{n} b_i.$$ 

Finally, if the constructed set does not belong to the same orbit, one may skewer and flip to assure that they do belong to the same orbit yielding the desired collection. 
\end{proof}

\begin{lemma}
Suppose that $X$ is an irreducible CAT(0) cube complex with a nonelementary and essential action of the group $\G$. Any nonempty subset $\frakH'\subset \frakH$ verifying the following properties must be equal to $\frakH$: 
 \begin{itemize}
\item \emph{(Symmetric):} $(\frakH')^*=\frakH'$;
\item \emph{($\G$-invariant):} $\G\cdot \frakH' = \frakH'$;
\item \emph{(Convex):} If $h, h' \in \frakH'$ with $h\subset h'$ and $k \in \frakH$ such that $h \subset k \subset h'$ then $k \in \frakH'$. 
\end{itemize}
\end{lemma}

\begin{proof}
 Since $X$ is irreducible, and $\frakH'$ is nonempty and $\G$-invariant, we can apply Lemmas \ref{extend4ss} and \ref{Double Skewering Lemma} to obtain a bi-infinite sequence of pairwise strongly separated half-spaces $\{h_n : n \in \Z\} \subset \frakH'$ with $h_{n+1}\subset h_n$.

 Let $k \in \frakH$. Then, there is at most one element of $\{h_n: n\in \Z\}$ which is transverse to $k$. This means that, there is an $N\in \Z$ for which $\^k \subset {h}^*_{N+2}\cap  h_N$. Since $\frakH'$ is symmetric and convex, we conclude  that $k\in\frakH'$. 
\end{proof}

\begin{cor}\label{cor:product}
 Assume we have an essential and nonelementary action of $\G$  on $X$, and $\G_0\leq \G$ of finite index. If $\frakH' \subset \frakH$ is a non-empty symmetric convex $\G_0$-invariant collection of half-spaces. Then either $\frakH'= \frakH$ or  $X \cong X' \times X''$ and $\frakH'$ is the half-space structure for $X'$.
\end{cor}

\section{The Furstenberg-Poisson Boundary}
We now assume that $\G$ is a  discrete countable group.

The interested reader should consult the following references for further details \cite{FurmanRW}, \cite{Kaimanovich}, \cite{BaderShalom}, \cite{CFI}, \cite{BaderFurman}. This exposition follows closely these sources, as well as a nice series of lectures by Uri Bader at CIRM in winter 2014.

\begin{definition}
Consider a measurable action $\a : \G \times M \to M$ of the group $\G$ on the measure space $(M,m)$ and $\mu$  a measure on $\G$.  The convolution as a measure on $M$ is  the push forward under the action map of the product measure from $\G\times M$: 
$$\mu*m = \a_*(\mu\otimes m).$$ 
\end{definition}

We shall make use of the following elementary fact:

\begin{lemma}\label{HaarConv}
 Let $Haar$ denote the counting measure on $\G$, $\delta_e$ the Dirac measure at the identity $e\in \G$ and $\mu \in \P(\G)$ be a probability measure. Then $Haar*\mu = Haar$, and $\delta_e*\mu = \mu= \mu*\delta_e$.
\end{lemma}

The proof of this is straightforward but we record it to exemplify the usefulness of thinking of convolution of measures in the context of pushforwards as above:

\begin{proof}
Let us show that $Haar*\mu(\g) = 1$ for every $\g \in \G$. Indeed,
 \begin{eqnarray*}
Haar*\mu(\g) &=& Haar\otimes \mu\{(\g_0, \g_1) : \g_0\g_1 = \g\}\\
&=& \Sum{\g_1 \in \G} Haar(\g\g_1^{-1}) \mu(\g_1) \\
&=& \Sum{\g_1 \in \G} \mu(\g_1)= 1
\end{eqnarray*}

A similar calculation shows that $\delta_e*\mu = \mu = \mu*\delta_e$.
\end{proof}

\begin{definition}
A probability measure $\mu \in \P(\G)$ is said to be generating if for every $\g \in \G$ there are $h_i \in \supp(\mu)$ such that $\g = h_1\cdots h_n$, i.e. the support of $\mu$ generates $\G$ as a semigroup.
\end{definition}

Given a generating measure $\mu$, we will associate two spaces to the $\mu$-random walk. The space of increments and the path space. As sets, these two spaces will be the same, but the measures on them will be different. 

Let $\G^\N=\{ \~\om=(\om_n)_{n\geq1}: \om_n\in \G\}$. The measure $\mu$ on $\G$ naturally induces a measure $\mu^\N$ on $\G^\N$ which assigns measure $\mu(g_1)\cdots\mu(g_n)$ to the  cylinder set:
$$C_{i_1, \dots, i_n}(g_1, \dots, g_n) = \{\~\om \in \G^\N : \om_{i_j} = g_j \text{ for } j = 1, \dots, n\}.$$ 

Let $\Om := \G \times \G^\N =\{(\om_0, \om_1, \dots) : \om_n \in\G\}$.
Given another measure $\theta$ on $\G$, which is not assumed to be a probability measure, we can consider the associated measure $\theta\otimes \mu^\N$ on $\Om$. This is the \emph{space of increments}, where we see the first factor as where to start the random walk (with distribution $\theta$). We will consider the action of $\G$ on $\Om$ which is transitive on the first factor and trivial on the rest. 

Next let $ \Om'= \G^\N$. We will consider the diagonal action of $\G$ on $\Om'$.
Observe that there is a natural map $W: \Om \to \Om'$, $(\om_0, \om_1, \om_2, \dots) \mapsto \~\om'$ where the $n$-th component of the image is given by  %
$$\om_n'= \om_0 \om_1\om_2\cdots \om_{n-1}.$$

With the actions of $\G$ defined above on $\Om$ and $\Om'$ we note that $W$ is $\G$-equivariant. 
We think of the image of this map as the space of sample paths. 
Consider the  time shift map:
$$S: \Om \to \Om; (\om_0,\om_1, \om_2 \dots) \mapsto (\om_0\om_1, \om_2, \om_3, \dots)$$
which is just a composition of the standard action map $\G\times \G \to \G$ given by $(\om_0, \om_1) \mapsto \om_0\om_1$ with the  time shift map:
$$S':  \Om'\to \Om'; (\om'_0,\om'_1, \om'_2\dots) \mapsto (\om'_1, \om'_2, \om'_3, \dots)$$
With these definitions in place, we observe that $W\circ S = S'\circ W$. 

Finally, applying Lemma \ref{HaarConv}, we deduce:
$$S_*(Haar\otimes \mu^\N) = (Haar*\mu)\otimes\mu^\N = Haar\otimes\mu^\N,$$
(i.e. that $S$ preserves $Haar\otimes \mu^\N$) and 
$$S_*(\delta_e\otimes \mu^\N) = (\delta_e*\mu)\otimes\mu^\N = \mu\otimes\mu^\N.$$

As it will be important below, we denote by $\mathbf{P} = \delta_e \otimes \mu^\N$ and $\mathbf{P}' = W_*\mathbf{P}$ the probability measures on $\Om$ and $\Om'$ respectively.

\begin{definition}
The space of ergodic components of the semi-group action generated by $S'$ on the space of sample paths $\Om'$ with measure class $W_*(Haar \otimes\mu^\N)$ is the \emph{Furstenberg-Poisson Boundary} for the $\mu$-random walk on $\G$ and will be denoted by $B$. Define the \emph{probability measure} $\nu$ on $B$ to be the push forward of $\mathbf{P}'$ under the natural projection $ \Om'\to B$.  
\end{definition}

Observing that $W_*(Haar \otimes\mu^\N)$ is $\G$-invariant and that the action of $\G$ commutes with the semigroup-action of $S'$, one sees that $\G$ must preserve the ergodic components of $S'$ and hence, the action of $\G$ descends to $B$. 

Furthermore, $S'$ preserves the measure $W_*(Haar\otimes \mu^\N)$ and $W_*(\delta_e\otimes \mu^\N)$ is absolutely continuous with respect to $W_*(Haar\otimes \mu^\N)$ so that $\nu$ is well defined on $B$. Finally, the following calculation  shows that $\mu *\nu = \nu$ and hence that $\nu$ is $\mu$ stationary:
\begin{eqnarray}
\mu*W_*(\delta_e \otimes \mu^\N) &=& W_*((\mu*\delta_e)\otimes \mu^\N) \nonumber\\
&=& W_*S_*(\delta_e\otimes\mu^\N) \nonumber\\
&=& S'_* W_*(\delta_e \otimes\mu^\N).  \nonumber
\end{eqnarray}

\begin{definition}
Let $\mu$ be a probability measure on $\G$. A  $\G$-equivariant measurable quotient of the Furstenberg-Poisson boundary $(B,\nu)$ is called a \emph{$(\G,\mu)$-boundary}.
\end{definition}

\section{Isometric Ergodicity}

Recall that a (quasi-measure-preserving) action of $\G$ on a measure space $(E,\nu)$ is said to be \emph{ergodic} if any $\G$-invariant Borel map $f: E\to \R$ is essentially constant.

The Furstenberg-Poisson boundary has very robust ergodicity properties. Bader and Furman  have developed a  general and powerful framework within which one can exploit these ergodicity properties in what is part of a great unification (and extension) program of previous super-rigidity results (see \cite{BaderFurman}).

Let $(E,\nu)$ be a Borel space on which the group $\G$ acts measurably and quasi-preserves the measure $\nu$. Such a space will be called a \emph{Lebesgue $\G$-space}. We say that the $\G$ action is \emph{isometrically ergodic} if the following holds: 

Let $(M,d)$ be a separable metric space and $\G\to \Isom(M,d)$ an action by isometries. If $f: E\to M$ is a $\G$-equivariant map, then it is essentially constant. 

We remark that, for an isometrically ergodic action, the existence of such a map $f$ is equivalent to the existence of a $\G$-fixed point in $M$. 

 Let $\M$ and $V$ be standard Borel spaces. We say that a Borel map $q: \M \to V$  is \emph{relatively metrizable} if there is a Borel map on the fibered product $d:\M\times_V\M \to [0,\8)$ such that the restriction $d_v$ to each fiber $M_v:=q^{-1}(v)$ is a separable metric. Such a Borel map $d$ is called a \emph{relative metric} on $q: \M \to V$. Furthermore, a \emph{relatively isometric action} of $\G$ on $q: \M \to V$ is a pair of $q$-compatible Borel-actions of $\G$ on $\M$ and $V$ so that as maps on fibers, each $\g\in \G$ is an isometry; that is, if $v\in V$ and  $x,y \in M_v$ then 
 $$d_{\g v}(\g x, \g y) = d_v(x,y).$$

\begin{definition}
 Suppose $B$ and $B'$ are Lebesgue $\G$-spaces. A $\G$-equivariant Borel map $p : B' \to B$ is said to be \emph{relatively isometrically ergodic} if for every relatively isometric action of $\G$ on $q: \M \to V$, and any $(p,q)$-compatible maps $u: B' \to \M$ and $\ell : B \to V$ there exists a map $f: B \to \M$ making the following diagram commute:
\begin{diagram}
B' & \rTo^u & M  \\
\dTo^{p} & \ruDashto^{\exists\,f} & \dTo_{q} \\
B & \rTo^\ell& V
\end{diagram}
%
\end{definition}

\begin{remark}
 We note that if one replaces the target spaces $B$ and $V$ with the one point space $\{*\}$ then one will recover the notions defined above in ``non-relative" terms. Namely, with respect to the one point projection relative metrizability  is just metrizability; a relatively isometric action is just an isometric action; and relatively isometrically ergodic is just isometrically ergodic.
\end{remark}

\begin{definition}
 Let $\G$ be a locally compact second countable group. A pair $(B_-, B_+)$ of $\G$-Lebesgue spaces forms a \emph{boundary pair} if the $\G$-action on both $B_\pm$ are amenable and the projections $B_-\times B_+\to B_\pm$  are relatively isometrically ergodic. 
\end{definition}

Recall that if a Lebesgue $\G$-space $B$ is amenable (in the sense of Zimmer) then given a compact metrizable space on which $\G$ acts by homeomorphisms, there is a $\G$-equivariant map $B\to \P(K)$.

The following is a strengthening of a result of Kaimanovich \cite{Kaimanovich}. We state it here for discrete countable groups and note that the same statement holds for a  locally compact second countable group under the additional assumption that the measure $\mu$ is ``spread out". 

\begin{theorem}\label{FPBoundaryisBoundary}\cite[Theorem 2.7, Remark 2.4]{BaderFurman} Let $\G$ be a  discrete countable group and $\mu\in \P(\G)$ a generating probability measure. Let $(B_-, \nu_-)$ and $(B_+, \nu_+)$ be the Furstenberg-Poisson boundaries for $(\G, \mu)$ and $(\G, \ch\mu)$, respectively. Then $B_-\times B_+$ is isometrically ergodic and $(B_-, B_+)$ is a boundary pair for $\G$ and any of its lattices. 
\end{theorem}

Of course, since we have stated the theorem for discrete countable groups, a lattice is necessarily a finite index subgroup. 

\begin{prop}\label{BaderFurmanProp}\cite[Proposition 2.2]{BaderFurman}
The property of relative isometric ergodicity is closed under composition of $\G$-maps.
\end{prop}

As a direct consequence of these, and by considering the second countable metric space $M = \ell^2(C)$, we have the following:

\begin{cor}\label{cor:EquivToCountable}
Let $C$ be a countable set on which $\G$ acts by permutations and $(P,\vartheta)$ an isometrically ergodic $\G$-space.  There is a $\G$-equivariant map $P \to C$ if and only if there is a $\G$-fixed point in $C$. In particular, this holds for $P= \P_-\times \P_+$ where $(\P_-,\P_+)$ is a $\G$-equivariant quotient of a $\G$-boundary pair $(B_-, B_+)$.
\end{cor}

\section{Tools}

As we saw in the previous section, a key characteristic of the Furstenberg-Poisson boundary is that it is Zimmer amenable. This connects the study of boundary maps to the study of probability measures on the space of interest. We now develop some background and tools for this purpose. More specifically, we would like to understand a probability measure on the Roller compactification $\~X$ from a geometric perspective.

\subsection{Measures on $\~X$}

Let $\P(\~X)$ denote the space of probability measures on $\~X$. If $m\in\P(\~X)$ define

\begin{eqnarray*}\label{eq:h_mu}
H_m:=\{h\in\frakH(X):m(h)=m(\*h)\}\\
H_m^+:=\{h\in\frakH(X):m(h)>1/2\}\\
H_m^-:=\{h\in\frakH(X):m(h)<1/2\}\\
H_m^\pm:=\{h\in\frakH(X):m(h)\neq1/2\}\.
\end{eqnarray*}

\begin{lemma}\cite[Lemmas 4.6, 4.7]{CFI}\label{Hmufacts} 
The maps $\P(\~X)\to 2^{\frakH(X)}$ given by ${m} \mapsto H_{m}, H_{m}^+, H_{m}^-$ are $\Aut(X)$-equivariant 
for the natural actions on $\P(\~X)$ and $2^{\frakH(X)}$. 
Furthermore, for $m,m'\in\P(\~X)$ the following hold:

\begin{enumerate}
\item\label{item:nofacingtriplesinhmu} There are no facing triples in $H_{m}$.  
If $X$ is not Euclidean then $H_{m}^+$ has facing triples and in particular $H_{m}^+\neq\varnothing$.
\item\label{item:involution}\label{HmuIsEuclidean} The collection $H_{m}$ is convex, involution invariant, and the associated complex $\~X(H_m)$ is an interval. 
\item\label{item:partition} The collection of half-spaces $H_m^+$ is consistent and yields a lifting decomposition 
$\frakH(X) = H_{m} \sqcup (H_{m}^+\sqcup H_m^-)$ and corresponds to the subcomplex denoted by $\~X_m\subset \~X$. Therefore $\~X_m\cong \~X(H_m)$.

\end{enumerate}
 \end{lemma}

\begin{remark}\label{DCC Halfspaces}
Fix $m\in \P(\~X)$. By Proposition \ref{LiftingDecomp}, it follows that $X_m\subset X$ whenever $H_m^+$ satisfies the descending chain condition. Furthermore, as in Remark \ref{Remark product intervals is interval in product}, if $X$ has an irreducible decomposition corresponding to $\frakH= \frakH_1\sqcup \cdots \sqcup \frakH_n$, then we have that $X_m\subset X$ whenever $H_m\cap \frakH_i$ satisfies the descending chain condition for each $i = 1, \dots, n$. 
\end{remark}

\subsection{Strong Separation and Measures on $\~X$}
As above, consider $X = X_1\times\dots \times X_n$ be the product decomposition of $X$ into irreducible factors. Let $\S_i\subset \frakH_i\times \frakH_i$ denote the pairs of disjoint strongly separated half-spaces in  $X_i$  and let $\S = \S_1\sqcup \cdots \sqcup \S_n$. With this notation in place, and recalling that an interval in a product is the product of the corresponding intervals (see Remark \ref{Remark product intervals is interval in product}) we have the following easy generalization of \cite[Lemma 4.18]{CFI}: 

\begin{lemma}\label{DCCHeavy}
 Let $X$ be a CAT(0) cube complex with product decomposition corresponding to $\frakH = \frakH_1\sqcup \cdots \sqcup \frakH_n$. If $m\in \P(\~X)$ and $(H_m\times H_m) \cap \S_i \neq \varnothing$ for each $i = 1, \dots, n$, then $H_m^+$ satisfies the descending chain condition. 
\end{lemma}

\begin{lemma}\label{comparing measures1}

Assume $X$ is an  irreducible CAT(0) cube complex with the action of $\Aut(X)$ nonelementary and essential. Let $m, m' \in \P(\~X)$ such that $H_m\cap H_{m'} \neq \varnothing$. Then, there exists strongly separated  half spaces $h, k \in \frakH$ such that:
 \begin{itemize}
\item $h\subset k$;
\item $h \in H_m^-\cap H_{m'}^-$ and $k\in H_m^+\cap H_{m'}^+$;
\item $\^x \in h^*\cap k$ for every $x\in H_m\cup H_{m'}$.
\end{itemize} 
\end{lemma}

\begin{proof}
By Lemma \ref{Hmufacts},  $H_m$ and $H_{m'}$ do not contain facing $n$-tuples for $n\geq 3$. Let $y\in H_m\cap H_{m'}$. Applying Lemma \ref{ssContaining x} for $n=6$, we find $k_1, \dots, k_6$ pairwise strongly separated facing half-spaces with 
$$ \^y \subset \Cap{i = 1}{6} k_i.$$
At most two of these belong to $H_m$ and similarly, at  most two belong to $H_{m'}$ meaning that at least two, say $k_1, k_2 \in H_{m}^\pm \cap H_{m'}^\pm$. Up to replacing $y$ by $y^*$ we have that $k_1^*\subset y\subset k_2$, and so $k_1, k_2 \in H_m^+\cap H_{m'}^+$.

By Lemma \ref{extend4ss} we find $h, k\in \frakH$  such that $h$ and $k_1^*$ are strongly separated, $k_2$ and $k$ are strongly separated, and 
$$h\subset k_1^*\subset y\subset k_2 \subset k$$
meaning that 
for $m'' \in \{m, m'\}$
$$m''(h)\leq m''( k_1^*)< m''(y) = 1/2< m''( k_2 )\leq m''( k).$$
In particular, $h\in H_m^-\cap H_{m'}^-$ and $k\in H_m^+\cap H_{m'}^+$. 

Once more, let $m''\in\{m,m'\}$  and let $x \in H_{m''}$. Since $m''(k_2)>1/2$ and $m''(k_1^*)<1/2$ it follows that $x \not \subset k_1^*$ and $k_2 \not\subset x$. This means that either $x\pitchfork k_1^*$ or $x\pitchfork k_2$ or $\^x \subset k_1\cap k_2$. Either way, we conclude that $\^x \subset h^*\cap k$. 
\end{proof}

\begin{lemma}\label{ssmeas}
Let $X$ be an irreducible essential and nonelementary  CAT(0) cube complex. 
 Let $E \subset \P(\~X)$ be nonempty and $m\in \P(\~X)$. Suppose that if  $m'\in E$ then $H_{m'}$ does not contain strongly separated pairs and  $H_m\cap H_{m'}\neq \varnothing$. Then, there exist a strongly separated pair $h, k \in \frakH$ such that $h\subset k$, and for every  $x\in\Cup{m' \in E}{} H_{m'}$
 $$\^x \subset h^*\cap k.$$
\end{lemma}

\begin{proof}
Fix $m_0\in E$. By Lemma \ref{comparing measures1} applied to the measures $m$ and $m_0$, it follows that there is a strongly separated pair $h_0, k_0 \in \frakH$ such that $h_0 \subset k_0$, $h_0 \in H_m^-$, $k_0\in H_m^+$ and $\^x \in h_0^* \cap k_0$ for every $x \in H_m$. 

Now we apply  Lemma \ref{extend4ss}, to find pairwise strongly separated half-spaces $h_2, h_1, k_1, k_2$ such that
$$h_2\subset h_1 \subset h_0\subset k_0 \subset k_1 \subset k_2.$$

Now, for each $m' \in E$, and  for each $x\in H_m\cap H_{m'}$, up to replacing $x$ by $x^*$ if necessary,
$$h_2\subset h_1 \subset h_0\subset x\subset k_0 \subset k_1 \subset k_2.$$
This means that for $m'\in E$ and $x\in H_m\cap H_{m'}\neq \varnothing$, we conclude that $m'(h_i) \leq 1/2$ and $m'(k_i) \geq 1/2$ for $i = 0, 1, 2$. Furthermore, since $h_0$ and $h_1$ are strongly separated, and $H_{m'}$ does not have strongly separated pairs, it follows that $h_1 \in H_{m'}^-$, and hence $h_2\in H_{m'}^-$. Similarly, we conclude that $k_1, k_2\in H_{m'}^+$. 

Now, let $y \in H_{m'}$. Then, by measure considerations as above $y\not \subset h_1 $ and $k_1\not \subset y$ and hence $y\pitchfork h_1$, $y\pitchfork k_1$ or $\^y \subset h_1^*\cap k_1$. In either case, we conclude that $\^y \subset h_2^*\cap k_2$. 

\end{proof}

\begin{definition}
Let $H\subset\frakH(X)$. An element $h \in H$ is called:
 \begin{itemize}
\item  {\em minimal in $H$} if for every $k\in H$ either $k \pitchfork h$, $h \subset k$,  
or $h \subset \*k$;
\item  {\em maximal in $H$} if for every $k\in H$ either $k \pitchfork h$, $k\subset h$, 
or $\*k \subset h$, 
that is to say, $h$ is maximal if $\*h$ is minimal;
\item {\em terminal in $\frakH'$} if it is either maximal or minimal.
\end{itemize}
\end{definition}

\begin{lemma}[\cite{CFI}]\label{Lem Terminal map measu}
The map $\tau: 2^{\frakH(X)}\to 2^{\frakH(X)}$ taking a collection of half-spaces to its possibly empty collection of terminal elements is measurable and $\Aut(X)$-equivariant. 
\end{lemma}

Let us look at some examples of sets of half-spaces that do and do not have terminal elements. Consider $x\in \~X$ and the associated Dirac mass $\delta_x$. Then, $U_x$, the collection of half-spaces that contain $x$, corresponds to the heavy half-spaces of $\delta_x$, namely $U_x= H_{\delta_x}^+$. Now, if $x\in X$ then $U_x$  satisfies the descending chain condition. This means exactly that $\tau(H_{\delta_x}^+) \neq \varnothing$. Furthermore, if $x\in X$ and $x$ belongs to infinitely many cubes (which may be the case if $X$ is not locally finite) then,  $\tau(H_{\delta_x}^+) $ is in fact infinite. On the other hand, if we set $X = \Z$ with the standard cubulation and take $m =\frac{1}{2}(\delta_{-\8} + \delta_{+\8})$, where $\partial \Z =\{-\8, +\8\}$, then we see that all half-spaces are balanced for $m$ and that $\tau(H_m)= \varnothing$.

We record the following straightforward but important fact:

\begin{remark}\label{rem: finite terminal Hm}
Recall that to each measure $m\in \P(\~X)$ the space associated to $H_m$ is $\~X(H_m)$ which is an interval and therefore  Euclidean. This means that, if $H\subset H_m$ is any subset, then it must have finitely many terminal elements. 
 
\end{remark}

\subsection{Space of Intervals}
As we saw above, it is natural to associate to a probability measure $m\in \P(\~X)$ an interval $\~X_m\subset \~X$. For this reason, let us now look at intervals from a more global perspective.  

 Consider the map $2^\frakH\times 2^\frakH \to 2^\frakH\times 2^\frakH$ where $(S, T) \mapsto (S\cap T, S\triangle T)$. The \emph{space of intervals} is the image under this map of $\~X\times \~X\subset 2^\frakH\times 2^\frakH$ and is denoted by  $\I(\~X)$. Once more,  for $v,w\in \~X$ we set $\frakH(v,w) = U_v\triangle U_w$, and $\frakH(v,w)^+ = U_v\cap U_w$. Then
  $$\I(\~X) = \{\(\frakH(v,w), \frakH(v,w)^+\): v, w\in \~X\}.$$ 
Observe that by Corollary \ref{Interval on Finitely Many}  the map $\~X\times \~X \to \I(\~X)$ is finite-to-one.

\begin{lemma}\label{Lem Intervals Borel}
 The space of intervals $\I(\~X)$ is closed and hence Borel as a subset of $2^\frakH \times 2^\frakH$. 
\end{lemma}

\begin{proof}
 Let $(S_n, S_n^+) \in \I(\~X)$ and assume that $(S_n, S_n^+) \to (S, T)\in 2^\frakH \times 2^\frakH$. By assumption, there exists $v_n, w_n \in \~X$ such that $S_n = \frakH(v_n, w_n)$ and $S_n^+ = \frakH(v_n, w_n)^+$. Since $\~X$ is compact, there is a subsequence so that $(v_{n_j} , w_{n_j}) \to (v,w) \in \~X\times \~X$.  We claim that $(S,T) = (\frakH(v, w), \frakH(v, w)^+)$. 
 
 Let $(h, k) \in \frakH(v, w)\times \frakH(v, w)^+$ and without loss of generality, assume that $w\in h$ and $v\in h^*$. Then, since $(v_{n_j} , w_{n_j}) \to (v,w)$, for $j$ sufficiently large we have that $w_{n_j}\in h$, $v_{n_j}\in h^*$ and $w_{n_j}, v_{n_j}\in k$. Since $(\frakH(v_n, w_n), \frakH(v_n, w_n)^+) \to (S,T)$ it follows that $(h,k) \in S\times T$. 

Conversely, if $(h,k) \in S\times T$ then for all $n$ sufficiently large, we have that $(h,k) \in \frakH(v_n, w_n)\times  \frakH(v_n, w_n)^+$. In particular this holds for $n=n_j$ sufficiently large and so $(h,k) \in \frakH(v, w)\times  \frakH(v, w)^+$.

\end{proof}

Let $\I(\~X, *)$ denote the collection of pointed intervals of $\~X$, it is defined to be the collection of elements $2^\frakH\times2^\frakH \times 2^\frakH$  
$$\{\( \frakH(v,w), \frakH(v,w)^+, U_x\): x\in \I(v,w),  v, w\in \~X\}.$$

The proof of Lemma \ref{Lem Intervals Borel} easily generalizes to show:

\begin{lemma}\label{markedNice}
 The subset of pointed intervals $\I(\~X, *)\subset 2^\frakH\times2^\frakH \times 2^\frakH$ is closed and hence Borel. 
 \end{lemma}

We will employ $\I(\~X, *)$  in the proof of Theorem \ref{BoundaryMap}. For the sake of simplicity, we will think of an element in $\I(\~X,*)$ rather than as triples of sets of half-spaces, as a pair $(\I,x)$ such that $\I$ is an interval of $\~X$ and $x\in \I$.

\begin{remark}\label{projection}
 We observe that there is a natural continuous projection $\I(\~X,*) \to \~X$ where $(\I,x) \mapsto x$. 
\end{remark}

\section{Boundary Maps}

In this section we will prove the existence of measurable equivariant maps between the Furstenberg-Poisson boundary and the Roller boundary. The existence of such maps, when the target has  a convergence action of the group in question is guaranteed by \cite[Theorem 3.2]{BaderFurman}.  However, as we shall see in Section \ref{Convergence Section}, the action on the Roller compactification is rarely a convergence action. 

In this section, we will prove the existence of such maps, and also study some consequences associated to them.

\subsection{Existence}

A similar theorem to the following is proved in \cite{CFI}. 

\begin{theorem}\label{BoundaryMap}
Let $\G$ be a discrete countable group, $(B_-, \nu_-)$ and $(B_+, \nu_+)$ be a boundary pair for $\G$. Assume that $\G \to \Aut(X)$ is an action by automorphisms on a CAT(0) cube complex $X$. Then there exists a subgroup $\G'$ of finite index in $\G$ and $\G'$-equivariant measurable maps $\f_\pm : B_\pm \to \~ X$. Furthermore, if the action of $\G$ is nonelementary then $\G' = \G$ and $\f_\pm(B_\pm) \subset \partial X$. 
  \end{theorem}

\begin{proof}
 If $\G$ has a finite orbit in $\~ X$ then there is a finite index subgroup $\G_0$ fixing a point $x_0\in \~ X$ and a $\G_0$-equivariant measurable map $B_\pm \to \{x_0\}$. 
 
 Now, suppose that $\G$ does not have a finite orbit in the Roller compactification $\~ X$. Then, by Proposition \ref{prop:visual-to-roller} there exists a finite index subgroup $\G'\leq \G$ and a subcomplex $X'\subset \~ X$ on which the   $\G'$-action is nonelementary and essential, with $\frakH(X') \subset \frakH$. Furthermore, if the $\G$ action was assumed to be nonelementary on $X$, then we have that $X'\subset X$ is $\G$-invaraint and essential.

We record the fact that we have possibly passed to a finite index subgroup and an invariant subcomplex which, a priori can belong to the Roller boundary. We also observe that $(B_-, B_+)$ continues to be a boundary pair for $\G'$ by Theorem \ref{FPBoundaryisBoundary}.  Finally, we  note that $\partial X' \subset \partial X$ is invariant under $\G'$. And so to conserve notation, we assume now that the $\G$-action itself is nonelementary and essential on $X$. 

By amenability of the $\G$ action on $B_\pm$ there exist  $\G$-equivariant maps  $ B_\pm \to \P(\~X)$. Let $\vartheta_\pm$ be the push forward under these maps of the measures $\nu_\pm$. To avoid confusion, let us denote by $\P_\pm$ the space  $\P(\~X)$ with the measure $\vartheta_\pm$. Then our goal now is to extract  maps $\f_\pm:\P_\pm \to \~X$ which are measurable, $\G$-equivariant, and defined on a co-null set. To this end, observe that by Proposition \ref{BaderFurmanProp} the projection maps $\pi_\pm:\P_-\times \P_+ \to \P_\pm$ are relatively isometrically ergodic.

Let $m \in \P(\~X)$ and recall that we have an associated interval  denoted by $\~X_{m}$ whose half-space structure corresponds to the $m$-balanced half-spaces $H_{m}$. As a subset of $\~X$ we have
$$\~X_m = \Cap{h\in H_{m}^+}{} h.$$

We will be dealing with various natural maps $\P(\~X)$ to  $2^\frakH$ or $\R$. We cite  \cite{CFI} for the measurability of all of them and do not address the issue again.

Consider the maps is $(m_-, m_+)\mapsto \#(H_{m_-}\cap H_{m_+} ),   \#(H_{m_-}\triangle H_{m_+} )$. Since these are $\G$-invariant, they must be essentially constant by ergodicity of $\P_-\times \P_+$. Hence, we must be in one of the following cases:

\vskip.2in
\noindent
\underline{{\tt{I.}}   $H_{m_-}\cap H_{m_+} = \varnothing$, for $\vartheta_-\otimes \vartheta_+$-a.e. $(m_-, m_+)$. }

Fix a generic $(m_-,m_+) \in \P_-\times \P_+$. Then, we must have that $H_{m_-}\subset H_{m_+}^\pm$ and $H_{m_+}\subset H_{m_-}^\pm$. By Lemma \ref{Hmufacts} both $H_{m_-}^+$ and $H_{m_+}^+$ are consistent sets. This means that the the following collections of half-spaces satisfy both consistency and totality, i.e. they correspond to points in the Roller compactification (which will be denoted by $p_{m_-}(m_+) \in \~X_{m_-}$ and $p_{m_+}(m_-)  \in \~X_{m_+}$ respectively):
$$U_{p_{m_-}(m_+) }=H_{m_-}^+\cup (H_{m_-}\cap H_{m_+}^+)$$
$$U_{p_{m_+}(m_-)  }= H_{m_+}^+\cup (H_{m_+}\cap H_{m_-}^+).$$ 

Recall that, as was developed at the end of Section \ref{subsec:intervals-median},  $\I(\~X, *)$, the collection of pointed intervals in $\~X$, is Borel and that there is a natural Borel map from $\I(\~X, *)$ to $\I(\~X)$ and $\~X$, obtained by ``forgetting" the additional information of the point, or interval, respectively. This gives rise to the following commutative diagram: 

\begin{diagram}
\P_-\times \P_+\; & \rTo&\I(\~X,*) &\rTo&  \~X \\
\dTo^{\pi_+} & \ruDashto^{\exists\,\psi_+
} & \dTo_{q} & &\\
\P_+ & \rTo&\I(\~X) && 
\end{diagram}

%
The lower horizontal map $\P_+
\to \I(\~X)$ corresponds to the map $$m_+ \mapsto \~X_{m_+},$$
and the upper horizontal map $\P_-\times \P_+ \to  \I(\~X,*)$ is 
$$ (m_-, m_+) \mapsto \( \~X_{m_+}, p_{m_+}(m_-)\).$$
We observe that the preimage $q^{-1}(\I) = \{(\I, x):x\in \I\}$ is a countable set and therefore, the map $d: \I(\~X,*)\times \I(\~X,*) \to [0,\8)$ defined by 
$$d\((\I_1, x_1),(\I_2, x_2)\) = 1-\delta(x_1, x_2)$$
clearly makes the preimage $q^{-1}(\I) $ into a separable metric space and so that $\G$ acts relatively isometrically on  $q:\I(\~X,*)\to \I(\~X)$.

Now, since the quotient $\P_-\times \P_+ \to \P_+$ is relatively isometrically ergodic for $\G$,  we deduce that there is a  measurable  $\G$-equivariant map defined on a conull set $\psi_+:\P_+ \to \I(\~X,*)$. The same argument with the obvious modifications yields $\psi_- :\P_- \to \I(\~X,*)$.  Post composing these with the map $\I(\~X,*)\to \~X$ we obtain
$$ \f_\pm: \P_\pm \to \~X.$$

To finish the proof, we will show that all other cases lead to a contradiction. Recall that we remain under the assumption that the $\G$ action on $X$ is essential and nonelementary. 

\vskip.2in
\noindent
\underline{{\tt{II.}}  $H_{m_-}= H_{m_+}\neq \varnothing$, for $\vartheta_-\otimes \vartheta_+$-a.e. $(m_-, m_+)$. }

Fix a generic $m_-\in \P_-$ and a $\vartheta_+$-conull and $\G$-invariant set $P_+ \subset \P_+$ so that $H_{m_+}= H_{m_-}$ for every $m_+\in P_+$. Set $\frakH' := H_{m_-}$, and observe that this is a nonempty, symmetric, and convex  $\G$-invariant set of half-spaces.  By Corollary \ref{cor:product}, it follows that either $\frakH'=\frakH$ or $X \cong X_{m_-}\times X_2$. This contradicts Corollary \ref{Cor.Not an interval}: The $\G$ action is essential and nonelementary so  $X$ can not have an interval as a factor. 

\vskip.2in
\noindent
\underline{{\tt{III.}}   $H_{m_-}\cap H_{m_+} \neq \varnothing$ and $H_{m_-}\triangle H_{m_+}\neq \varnothing$, for $\vartheta_-\otimes \vartheta_+$-a.e. $(m_-, m_+)$.}

Let $X= X_1\times \cdots \times X_n$ be the decomposition of $X$ into irreducible factors, $\frakH = \frakH_1 \sqcup \cdots \sqcup\frakH_n$ be the corresponding decomposition of half-spaces into pairwise transverse collections, and $\G_0$ be a normal finite index subgroup whose image is in $\Aut(X_1) \times \cdots \times \Aut(X_n)$. 
Recall that $\S_i\subset \frakH_i\times \frakH_i$ denotes the pairs of disjoint strongly separated half-spaces in the irreducible factor $X_i$  and  $\S = \S_1\sqcup \cdots \sqcup \S_n$. Observe that $\S_i$ is $\G_0$-invariant for each $i= 1, \dots, n$ and that $\S$ is $\G$-invariant. We then have that the map $m \mapsto \#\((H_m\times H_m)\cap \S_i\)$ is measurable, $\G_0$-invariant, and hence essentially constant for both $\vartheta_-$ and $\vartheta_+$.

Up to changing the roles of $\vartheta_-$ and $\vartheta_+$  we have the following two cases corresponding to the essential values of $m \mapsto \#\((H_m\times H_m)\cap \S_i\)$ for $i \in \{1, \dots, n\}$: Either the $\vartheta_+$-essential value is zero for some $i\in \{1, \dots, n\}$, or  the $\vartheta_-$ and $\vartheta_+$ essential values are both nonzero for every $i \in \{1, \dots, n\}$. 

\vskip.2in
\noindent
\underline{{\tt{III.a}} $(H_{m_+}\times H_{m_+})\cap \S_i = \varnothing$ for some $i$ and $\vartheta_+$-a.e. $m_+\in \P_+$.}

Let us assume that $i= 1$. Fix a $\vartheta_-$-generic $m_- \in \P_-$ and a $\vartheta_+$-conull $\G_0$-invariant set $P_+\subset \P_+$ such that the following hold for every $m_+\in P_+$:

\begin{itemize}
\item $H_{m_-}\cap H_{m_+} \neq \varnothing$
\item $H_{m_-}\triangle H_{m_+}\neq \varnothing$
\item  $(H_{m_+}\times H_{m_+})\cap \S_1 = \varnothing$ 
\end{itemize}

Consider now the $\G_0$-equivariant projection $\~X\to \~X_1$ which induces a $\G_0$-equivariant map $\P(\~X) \to \P(\~X_1)$. Recall  that the $\G_0$ action remains essential (Remark \ref{Finite Index Ess}) and nonelementary (Lemma \ref{lem:boundary of factors}) on each irreducible factor, in particular, on $X_1$.

Set $E$ and $m$ to be the push forwards of $P_+$ and $m_-$ respectively   under the $\G_0$-equivariant projection $\~X\to \~X_1$. Observe that $E$ is $\G_0$-invariant   and  that the above assumptions on $P_+$ and $m_-$ descend to $E$ and $m$ respectively. In particular,  the hypotheses of Lemma \ref{ssmeas} are satisfied. This means that there exists a strongly separated pair $(h,k) \in \S_1$ such that $h\subset k$, and for every $x\in \Cup{m'\in E}{}H_{m'}$ we have that
$$\^x \subset h^*\cap k.$$
By $\G_0$-invariance of $E$, it follows that $h^*$ and $k$ are not $\G_0$ flippable, which contradicts the Flipping Lemma \ref{flip}, as the action of $\G_0$ is essential and nonelementary on $X_1$. 

\vskip.2in
\noindent
\underline{{\tt{III.b}} For each $i$ we have  $(H_{m}\times H_{m})\cap \S_i \neq \varnothing$ for $\vartheta_\pm$-a.e. $m\in \P_\pm$. }
 
Fix a generic $(m_-,m_+) \in \P_-\times \P_+$. 
 In this case, by Lemma \ref{DCCHeavy} we must have that $H_{m_-}^+\cup H_{m_+}^+$ satisfies the descending chain condition, i.e. every descending chain has a terminal element. Furthermore,  our assumption  that $H_{m_-}\triangle H_{m_+} \neq \varnothing$ implies that 
 $$\(H_{m_-}\cap H_{m_+}^+\) \cup \(H_{m_+}\cap H_{m_-}^+\)\neq \varnothing.$$ 
 Hence, as a subset of $H_{m_-}\cup H_{m_+}$, as in Remark \ref{rem: finite terminal Hm}, there are finitely many terminal elements in $\(H_{m_-}\cap H_{m_+}^+\) \cup \(H_{m_+}\cap H_{m_-}^+\)$ and there is at least one because these are nonempty subsets of $ H_{m_-}^+\cup H_{m_+}^+$. This yields a $\G$-equivariant map from $\P_-\times \P_+$ to the countable collection of finite subsets of $\frakH$ and so by Corollary \ref{cor:EquivToCountable}, there is a finite set $\mathcal F\subset \frakH$ which is $\G$-invariant. But Lemma \ref{lem:NoFiniteColInvHalf} shows that this is incompatible with our assumption that the action is both essential and nonelementary.

\end{proof}

Recall that the Furstenberg-Poisson Boundaries associated to $\ch\mu$ and $\mu$ and (for generating $\mu \in \P(\G)$) give a boundary pair for a group $\G$ (see Theorem \ref{FPBoundaryisBoundary}) and hence by Theorem \ref{BoundaryMap} 
  we deduce:

\begin{cor}\label{PFBoundaryMap}
Let $\G$ be a discrete countable group and $\mu \in \P(\G)$ a generating probability measure. Suppose furthermore that  $\G\to \Aut(X)$ is a nonelementary and essential action  action on the CAT(0) cube complex $X$. Then there exist quasi-$\G$-invariant probability measures $\lambda_\pm \in \P(\partial X)$ so that $(\partial X, \lambda_-)$ and $(\partial X, \lambda_+)$ are $(\G,\ch\mu)$ and $(\G,\mu)$-boundaries, respectively. 
\end{cor}

\subsection{The Image and Regular Points}
Nevo and Sageev refine the description of the Furstenberg-Poisson boundary by passing from the full Roller boundary to the closure of the non-terminating elements \cite{NevoSageev}. In this section we give a further refinement in terms of the regular points in the Roller boundary, along with some corollaries. 

\begin{definition}
 Let $X$ be an irreducible  CAT(0) cube complex. Define $\partial_{\mathrm{r}} X $, the \emph{regular points} as the set of $\xi\in \partial X$ such that if $h_1, h_2\in U_\xi$ then there is a $ k \in \frakH $ such that $  k\subset h_1\cap h_2$ and $k$ is strongly separated from both $h_1$ and $h_2$. If $X$ is reducible, define the regular points to be the product of the regular points in each factor, i.e. $$\partial_{\mathrm{r}} X = \partial_{\mathrm{r}} X_1\times\cdots \times \partial_{\mathrm{r}} X_n.$$ 
\end{definition}

Note that $\partial_{\mathrm{r}} X$ could be empty. Consider for example the connected complex obtained by removing the second and fourth quadrants in the plane. This example is admittedly degenerate having only finitely many automorphisms. 

\begin{prop}\label{Rank1Char}
Let $X$ be an \emph{irreducible} CAT(0) cube complex and $\a \in \partial X$. The following are equivalent:
 \begin{enumerate}
\item $\a\in \partial_{\mathrm{r}} X$;
\item There exists an infinite descending chain $\{s_n\}_{n\in \N} \subset U_\a$ of pairwise strongly separated half-spaces.
\item There exists a bi-infinite descending chain $\{s_n\}_{n\in \Z} \subset U_\a$ of pairwise strongly separated half-spaces.
\end{enumerate}
\end{prop}

\begin{proof}

We begin by observing that if $\{s_n\}$ is an infinite descending chain of strongly separated half-spaces and $h$ is a half-space whose intersection with each $s_n$ is nontrivial, then for $n$ sufficiently large we must have that $s_n\subset h$. By strong separation, we may assume that $h$ is parallel to each $s_n$. Now, if $s_1\subset h$ then we are done. Otherwise, $s_1\cap h^*\neq \varnothing$ and since $s_1 \cap h$ is nonempty as well,  by  Remark \ref{shortcut},  we have that $\^h\subset s_1$. Since $h\cap s_n\neq \varnothing$ then $s_n$ is not contained in $h^*$ for any $n$. Finally, since there are finitely many half-spaces in between any two, we conclude that for $n$ sufficiently large, $s_n\subset h$.

\noindent
 $(1) \implies (2)$:  Let $\a\in \partial_{\mathrm{r}} X$. Fix $s_1\in U_\a$. Then there exists $s_2\in U_\a$ such that $s_2\subset s_1 $ and $s_2$ and $s_1$ are strongly separated.

Assume that $s_1, \dots, s_n\in U_\a$  are decreasing and pairwise strongly separated. Since $\a\in \partial_{\mathrm{r}}X $ there is $s_{n+1}\in U_\a$ such that $s_{n+1}\subsetneq s_n\cap s_{n-1}$ with $s_{n+1}$ strongly separated with $s_{n}$.

\noindent
 $(2) \implies (1)$: This is straightforward. Assume $\{s_n\}\subset U_\a$ is an infinite descending sequence of pairwise strongly separated half-spaces and $h_1, h_2\in U_\a$. As was observed in the beginning of this proof, for $n$ sufficiently large,  
 $$s_{n+1}\subset s_n\subset h_1\cap h_2.$$
 $(2) \implies (3)$: Let $\{s_n: n\in \N\}\subset U_\a$ be an infinite descending sequence of pairwise strongly separated half-spaces. By the Double Skewering Lemma \ref{Double Skewering Lemma}, there exists $\g$ such that $\g s_1\subset s_2 \subset s_1$, that is, $s_1\subset \g^{-1}s_2 \subset \g^{-1} s_1$ and setting $s_{-n} = \g^{-n} s_1$ completes the desired sequence.
 
 \noindent
 $(3) \implies (2)$: This is trivial.
\end{proof}

We note that  conditions (1) -- (3) of Proposition \ref{Rank1Char} imply that $\a\in \partial_{NT} X$. That this is true for (1) is immediate from the definition of a regular point for irreducible complexes. That this is true for condition (2) and (3) follows as well: if $s_n, h\in U_\a$ then for $n$ sufficiently large we must have that $s_n \subset h$.

\begin{cor}\label{distinct} Let $X$ be irreducible.
The intersection of any infinite descending chain of strongly separated half-spaces, is a singleton. In particular, if $\xi_1, \xi_2\in \partial_{\mathrm{r}}X$ are distinct, then  $\xi_1 \in h_1$ and $\xi_2\in h_2$ for some strongly separated disjoint pair $h_1, h_2\in \frakH$.
\end{cor}

\begin{proof}
Let us show that if $\{s_n\}$ is a strongly separated descending chain of half-spaces then $\Cap{n\in \N}{} s_n$ is a singleton. Indeed, since every finite intersection of these half-spaces is non-empty, and $\~X$ is compact, $\Cap{n\in \N}{} s_n$ is nonempty. Suppose that $x,y \in \Cap{n\in \N}{} s_n$ are distinct. Then for some $h$ we have that $x\in h$ and $y\in h^*$. This of course means that for each $n$, $h\cap s_n$ and $h^*\cap s_n$ are both non-empty. By Remark \ref{shortcut}, we must have that for each $n$, either $h\pitchfork s_n$ or  $\^h \subset s_n$. By strong separation, it follows that $\^h \subset s_n$ for all $n$ sufficiently large. But this is impossible since $s_n$ is descending and there are finitely many half-spaces in between any two. Therefore, no such $h$ exists and $x= y$. 
\end{proof}

\begin{cor}
 If $X$ is a CAT(0) cube complex and the action of $\Aut(X)$ is essential and nonelementary, then $\partial_{\mathrm{r}}X\neq \varnothing$.
\end{cor}

\begin{proof}
Recall that the action of $\Aut(X)$ is essential and nonelementary if and only if the action of $\Aut(X_i)$ is essential and nonelementary for each irreducible factor $X_i$ of $X$ (Lemma \ref{lem:boundary of factors} and Theorem \ref{EssNonEmpty}). Caprace and Sageev's Theorem \ref{th:ss} characterizes such irreducible complexes by the existence of strongly separated pairs $s_1\subset s_0$ in $\frakH_i$. Applying the Double Skewering Lemma \ref{Double Skewering Lemma}, we find $\g s_1 \subset \g s_0 \subset s_1 \subset s_0$ and $\g s_1$ is strongly separated from $s_1$. Setting $s_n = \g^{n-1} s_1$ we obtain an infinite descending strongly separated chain and by Proposition \ref{Rank1Char}, we have that $\partial_{\mathrm{r}} X_i$ is nonempty and hence $\partial_{\mathrm{r}} X$ is nonempty.
\end{proof}

\begin{theorem}\label{minimal target}
Let $\G$ be a discrete countable group, $(B_-, \nu_-)$ and $(B_+, \nu_+)$ be a boundary pair for $\G$. Assume that $\G \to \Aut(X)$ is an essential and nonelementary action by automorphisms on a CAT(0) cube complex $X$. Then any $\G$-equivariant measurable map $\f_\pm : B_\pm \to \partial X$ has essential target in $\partial_{\mathrm{r}} X$. 

\end{theorem}

This will follow immediately from Theorem \ref{minimalRankFullMeas}. The rest of this section is devoted to proving this and other key results.  We first establish a bit of notation: for $x,y\in \~X$ the collection of half-spaces containing $y$ and not $x$ will be denoted by $[x,y]$. In terms of $U_x$ and $U_y$ we have $[x,y] = U_y\setminus U_x$.

\break

\begin{lemma}\label{InfiniteEssvalueOnS}
Assume $\G\to \Aut(X)$ is a nonelementary and essential action,  $\lambda_\pm$ quasi-$\G$-invariant measures on $\partial X$ such that $(\partial X^2, \lambda_-\otimes \lambda_+)$ is isometrically ergodic. Let  $N>0$ and $S \subset \frakH^N\times \frakH$ be a $\G$-invariant collection of $(N+1)$-tuples such that there exists $(h_1, \dots, h_N,k)\in S$ with 
$$h_1\subset \cdots \subset h_N\subset k^*.$$ 
The map $\partial X^2\to \N\cup \{\8\}$ defined by $(\xi_-, \xi_+) \mapsto\#\( [\xi_+, \xi_-]^N\times [\xi_-, \xi_+] \cap S\)$ is  $\lambda_-\otimes \lambda_+$-essentially constant with infinite essential value. 
 \end{lemma}

\begin{proof}
The measurability of the map in question relies on the $\G$-invariance of the non-empty set $S$. The proof is straight forward and similar to that of \cite[Corollary A.2]{CFI}. 

Since $S$ is  $\G$-invariant, it follows that the map in question is $\G$-invariant and hence essentially constant by (isometric) ergodicity. If the essential value is finite and non-zero, then this gives a $\G$-equivariant map from $\partial X^2$ to the countable collection of finite subsets of $\frakH$. By isometric ergodicity and Corollary \ref{cor:EquivToCountable} this yields a finite $\G$-invariant set in $\frakH$ which contradicts the assumption that the action is essential and nonelementary by Lemma \ref{lem:NoFiniteColInvHalf}. Therefore the essential value must be 0 or $\8$.

Fix $(h_1, \dots, h_N,k) \in S$ with $h_1\subset\cdots \subset h_N\subset k^*$. It follows that if $(\xi_-, \xi_+) \in h_1\times k$ then $(h_1, \dots, h_N,k) \in [\xi_+, \xi_-]^N\times [\xi_-, \xi_+] $, and since $\lambda_-\otimes \lambda_+(h_1\times k) >0$ by  Lemma \ref{PosMeasHalfSpaces}, we have that the essential value is not zero and hence infinite. 
\end{proof}

From this we derive some important consequences:

\begin{lemma} With the hypotheses as in Lemma \ref{InfiniteEssvalueOnS},
for $\lambda_-\otimes \lambda_+$ a.e. $(\xi_-, \xi_+) \in \partial X^2$ we have that $\I(\xi_-, \xi_+)\cap X\neq \varnothing$.
\end{lemma}

\begin{proof}
Let $X = X_1\times \cdots \times X_n$ be the decomposition of $X$ into irreducible factors and let $\G_0$ be the finite index subgroup of $\G$ which maps to $\Aut(X_1) \times \cdots \times \Aut(X_n)$. As before, we let $\S_i \subset \frakH_i\times \frakH_i$ be the collection of disjoint strongly separated pairs of half-spaces in the irreducible factor $X_i$, and $\S = \S_1\sqcup \cdots\sqcup \S_n$. Note that $\S_i$ is $\G_0$ invariant for each $i$.

Now, by assumption that the $\G$ action is essential and nonelementary on $X$, it follows that the quotient action of $\G_0$ on each $X_i$ is also essential and nonelementary (see Remark \ref{Finite Index Ess} and Lemma \ref{lem:boundary of factors}). By Caprace and Sageev's Theorem \ref{th:ss}, $\S_i\neq \varnothing$ and therefore, there is an $(h_i, k_i)\in \S_i$ and $h_i \cap k_i = \varnothing$. Observing that all the hypotheses remain true for $\G_0$, we may  apply Lemma \ref{InfiniteEssvalueOnS}, to the action of $\G_0$ on $S = \S_i$ for each $i=1, \dots, n$, and we deduce that, the essential value of $(\xi_-, \xi_+) \mapsto\#\( [\xi_+, \xi_-]\times [\xi_-, \xi_+] \cap \S_i\)$ is infinite, for each $i$. 

Next, as is observed in Remark \ref{DCC Halfspaces}, it follows that $X_m\subset X$, where $m$ is the average of the Dirac masses at $\xi_1$ and $\xi_2$. 
Of course, $\I(\xi_-, \xi_+) = \~X_m$ and hence $\I(\xi_-, \xi_+)\cap X\neq \varnothing$. 
\end{proof}

Consider two disjoint half-spaces $h$ and $k$ and define the map
 $$\delta(h,k) = \#\{\ell: h \subseteq \ell \subseteq k^*\}.$$
We note that this is not a distance on half-spaces,\footnote{ However, if $h$ and $k$ are disjoint and strongly separated, and $x\in h$ and $y \in k$ are as in Corollary \ref{Cor unique endpoints bridge} then $\delta(h,k) = d(x,y)$.}  although it is true that if $h\subsetneq \ell \subseteq k^*$ then $\delta(h,\ell^*) \leq \delta(h,k)$. %

\begin{definition}
 Suppose $\frakH= \frakH_1\sqcup\cdots \sqcup \frakH_n$ corresponds to the irreducible factor decomposition of $X$. Let $\mathcal S^{(N)}_R$ denote the collection of $(h_1, \dots, h_N, k) \in \frakH^{N+1}$ such that $h_1\subsetneq \cdots \subsetneq h_{N}\subsetneq k^* $, $\delta(h_1, k)\leq R$, and for some $i\in\{1, \dots, n\}$ 
$$(h_1, h_{2}), \dots, (h_{N-1}, h_N), (h_N, k)\in \mathcal S_i.$$

\end{definition}

\begin{lemma}\label{This good enough}
Assume the hypotheses as in Lemma \ref{InfiniteEssvalueOnS}. For each $N$ there is an $R$ such that for  each $i = 1, \dots, n$ and $\lambda_+\otimes \lambda_-$-a.e. $(\xi_-, \xi_+) \in \partial X\times \partial X$ the cardinality of $\mathcal   [\xi_+, \xi_-]^{N}\times [\xi_-, \xi_+] \cap \S^{(N)}_R\cap \frakH_i^{N+1}$ is infinite.
\end{lemma}

\begin{proof}
Let $\G_0$ be the finite index subgroup which preserves the irreducible factor decomposition $X= X_1\times \cdots \times X_n$ and note that the $\G_0$-action is still essential and non-elementary (Remark \ref{Finite Index Ess} and Lemma \ref{lem:boundary of factors}). Then, the result follows by applying Lemma \ref{InfiniteEssvalueOnS} to the $\G_0$ action on   $S=\S^{(N)}_R\cap \frakH_i^{N+1}$. We must therefore show that for each $N$ there is an $R$ so that for each $i= 1, \dots, n$, this collection is  not empty. 

Fix $i$, and  $h_i\subset k_i^*$ with $(h_i, k_i )\in \S_i$. Applying the Double Skewering Lemma \ref{Double Skewering Lemma}, we find $\g_i \in \G_0$ such that $\g_i k_i^*\subset h_i$. In particular, $\g_i h_i\subset h_i$ are strongly separated in $\frakH_i$.
Setting $R = \max{i =1, \dots, n} \delta(\g_i^{N-1} h_i, k_i)$, we have that $(\g_i^{N-1}h_i, \dots, h_i, k_i)\in \mathcal S^{(N)}_R\cap \frakH_i^{N+1}$. 
 
 \end{proof}

Recall that $\tau : 2^\frakH \to 2^\frakH$ measurably assigns to a set its terminal elements (Lemma \ref{Lem Terminal map measu}). If $H\subset \frakH^N$, by abuse of notation, we shall use $\tau(H)$ to denote the terminal elements in the union of the projections of $H$ to each factor. Namely, if the $i$-th projection is $p_i : \frakH^N \to \frakH$ and $H\subset \frakH^N$ then  
$$\tau(H): = \tau\(\Cup{i = 1}{N} p_i(H)\).$$
By Remark \ref{rem: finite terminal Hm}, (and by considering the average of the Dirac masses at two points $x,y\in \~X$) 
there are finitely many terminal elements in any subset of  $[x,y]\cup[y,x]$ and hence, by Corollary \ref{cor:EquivToCountable} we deduce:

\begin{cor}\label{FiniteTermElemntsN-Tuples}
Assume the hypotheses  in Lemma \ref{InfiniteEssvalueOnS} and let $\frakH= \frakH_1\sqcup\cdots \sqcup \frakH_n$ be the irreducible factor decomposition. For each $N$ and each $i=1, \dots, n$
 the following has $\lambda_-\otimes \lambda_+$-essential value zero: 
 $$(\xi_-, \xi_+)\mapsto \#\(\tau( [\xi_+, \xi_-]^{N}\times [\xi_-, \xi_+]\cap \S_R^{(N)}\cap\frakH_i^{N+1})\).$$
\end{cor}

We will denote by $\Delta$ the \emph{fat diagonal} in $\partial_{\mathrm{r}} X^2$. Namely, if $X= X_1\times \cdots \times X_n$ is the irreducible decomposition of $X$ then $\Delta$ is the  collection $((\xi^1_1, \dots \xi^n_1),(\xi^1_2, \dots \xi^n_2)) \in \partial_{\mathrm{r}}X^2$ such that $\xi_1^i =\xi_2^i$, for some $i$.

Theorem \ref{minimal target} is a corollary to:

\begin{theorem}\label{minimalRankFullMeas}
 With the hypotheses as in Lemma \ref{InfiniteEssvalueOnS}, it follows that 
 $$\lambda_-\otimes \lambda_+(\partial_{\mathrm{r}}X^2)= 1.$$
 and
 $$\lambda_-\otimes \lambda_+(\Delta)= 0.$$
\end{theorem}

\begin{proof}
Let $X= X_1\times \cdots \times X_n$  be the irreducible factor decomposition of $X$ and let $\G_0$ be the finite index subgroup of $\G$ which preserves each factor. Then, applying Lemma \ref{This good enough} and Corollary \ref{FiniteTermElemntsN-Tuples} to the  $\G_0$ action, and setting $\S_i({R}) = (\S_i)^{(3)}_{R}$  we deduce that, as maps  $\partial X\times \partial X \to \N\cup\{\8\}$, the essential value of
$$(\xi_-, \xi_+)\mapsto \#\([\xi_+, \xi_-]\times [\xi_-, \xi_+]\cap \S_i(R)\)$$
and 
$$(\xi_-, \xi_+)\mapsto \#\(\tau([\xi_+, \xi_-]\times [\xi_-, \xi_+]\cap \S_i(R))\)$$
is $\8$ and $0$ respectively. 

We claim that if $R>0$ then $\partial_{\mathrm{r}} X^2$ 
 contains the intersection of 
$$
\Cap{i=1}{n}\{(\xi_-, \xi_+)\in \partial X\times \partial X:\#\( [\xi_+, \xi_-]\times [\xi_-, \xi_+] \cap \S_i(R)\) = \8 \}
$$
with 
$$
\Cap{i=1}{n}\{(\xi_-, \xi_+)\in \partial X\times \partial X:\#\(\tau( [\xi_+, \xi_-]\times [\xi_-, \xi_+] \cap \S_i(R))\) = 0 \},
$$
which for $R$ sufficiently large, has  full $\lambda_-\otimes \lambda_+$-measure.

Suppose that $(\xi_-, \xi_+)$ 
is such that for some $R>0$ and for each $i= 1, \dots, n$,
$$\#\( [\xi_+, \xi_-]\times [\xi_-, \xi_+] \cap \S_i(R)\) = \8,$$
and
$$\#\(\tau( [\xi_+, \xi_-]^N\times [\xi_-, \xi_+] \cap \S_i(R)\)) = 0.$$
From these hypotheses, we will now construct a bi-infinite descending chain $\{s_{m}:m\in \Z\}\subset [\xi_+, \xi_-]\cap \frakH_i$ which, as elements of $\frakH_i$ are strongly separated. Proposition \ref{Rank1Char} and the definition of regular points (in the reducible case) complete the proof that $\lambda_-\otimes \lambda_+(\partial_{\mathrm{r}}X^2)= 1.$

Fix $i$. Suppose that $a_{m} \subsetneq a_{m-1}$, for $m\in \Z$ is a bi-infinite chain in $[\xi_+, \xi_-]$ such that for each $m$ there exists $b_m, c_m\in [\xi_+, \xi_-]$ such that $a_m \subsetneq b_m \subsetneq c_m$ are pairwise strongly separated in $\frakH_i$ with $\delta(a_m, c_m^*) \leq R$. We claim that $s_m:= a_{mR}$ is a bi-infinite descending chain, pairwise strongly separated in $\frakH_i$.

Fix $m$. Observe that $\delta(a_m, a_{m-R}^*)\geq R+1$. Since $a_m\subset c_m$ and $\delta(a_m, c_m^*) \leq R$ it must be that  $a_{m-R} \not\subseteq c_{m}$ i.e. $c_{m}^* \cap a_{m-R}\neq \varnothing$. Also, as  $a_{m-R} \cap c_m$ and $a_{m-R}^* \cap c_m^*$ contain $\xi_-$ and $\xi_+$ respectively we deduce   that either  $ c_{m} \subset a_{m-R}$ or $c_m \pitchfork a_{m-R}$. Either way, $b_m\subset a_{m-R}$ and hence $a_m$ and $a_{m-R}$ are strongly separated.

We now show that  $\lambda_-\otimes \lambda_+(\Delta) = 0$. Indeed, by  Corollary \ref{distinct}, it follows that $\Delta$ is contained in the union of the measure 0 set
$$\Cup{i= 1}{n}\{(\xi_-, \xi_+)\in \partial X\times \partial X:\#\( [\xi_+, \xi_-]\times [\xi_-, \xi_+] \cap \S_i(R)\) = 0\}.$$
\end{proof}

\subsection{Bridge Points}\label{Section Bridge Points}

Recall that as in Lemma \ref{ss Points are median}, associated to a strongly separated pair $h\subset k^*$ there is a combinatorial bridge $B(h, k)$ with the property that if $p \in B(h,k)$ then $p = m(x,p,y)$ for every $(x,y) \in h\times k$. 

\begin{definition}
Assume that $X = X_1 \times \dots \times X_n$ is the irreducible decomposition of $X$ into irreducible factors, corresponding to the decomposition $\frakH = \frakH_1\sqcup \cdots \sqcup \frakH_n$.  An element $x\in X$ is called a \emph{bridge point} if for each $i$, there exists a pair of disjoint half-spaces $h_i, k_i$, strongly separated in $\frakH_i$ such that $x \in B(h_1, k_1)\times \cdots \times B(h_n, k_n)$. 
\end{definition}

We note that the property of being a bridge point is $\G$-invariant.

\begin{lemma}\label{StripNonempty}
Assume $\G\to \Aut(X)$ is a nonelementary and essential action,  $\lambda_\pm$ quasi-$\G$-invariant measures on $\partial X$ such that $(\partial X^2, \lambda_-\otimes \lambda_+)$ is isometrically ergodic. There are disjoint half-spaces $h_i,k_i$ strongly separated in $\frakH_i$ such that for every bridge point $x\in B(h_1,k_1) \times \cdots \times B(h_n,k_n)$ and  $\lambda_-\otimes \lambda_+$-a.e. $(\xi_-, \xi_+) \in \partial _MX^2$ the map $(\xi_-, \xi_+) \mapsto \#(\I(\xi_-, \xi_+) \cap \G\cdot x)$ is  infinite.
\end{lemma}

Before proceeding with the proof, we note the straightforward but important fact that an interval in a product is just the product of the corresponding intervals (see Remark \ref{Remark product intervals is interval in product}).

\begin{proof}[Proof of Lemma \ref{StripNonempty}]

For every $x\in X$, the map $(\xi_-, \xi_+) \mapsto \#(\I(\xi_-, \xi_+) \cap \G\cdot x)$ is $\G$-invariant and hence essentially constant by  ergodicity. Furthermore,  $\G\cdot x \subset X$  is countable, and so the essential value must be 0 or $\8$ by Corollary \ref{cor:EquivToCountable}. 

Now, by Proposition \ref{Rank1Char}, Corollary \ref{distinct}, it follows that 
$$\partial_{\mathrm{r}}X^2 \setminus \Delta \subset \underset{i=1, \dots, n}{\underset{(h_i, k_i) \in \S_i}{\bigcup} }(h_1\times \cdots \times h_n) \times (k_1\times \cdots \times k_n).$$ 
By Theorem \ref{minimalRankFullMeas} this union has full measure and hence one of the sets must have positive measure. 
Fix $(h_1\times \cdots \times h_n) \times (k_1\times \cdots \times k_n)$ of positive measure   and note that it is chosen precisely to have a positive measure intersection with $\partial_{\mathrm{r}}X^2\setminus\Delta$. Let $x\in (h_1^*\cap k_1^*)\times \cdots \times (h_n^*\cap k_n^*)$ be a bridge point. Then, by Lemma \ref{ss Points are median}, the map $(\xi_-, \xi_+) \mapsto \#\(\I(\xi_-, \xi_+) \cap \G\cdot x\)$  takes non-zero values on this positive measure set and hence the map has infinite essential value.

\end{proof}

\begin{cor}\label{cor:StripsNonempty}
With the hypotheses and notation as in Lemma \ref{StripNonempty},  the strip $S: \partial X\times \partial X \to 2^\G$  given by $S(\xi_-, \xi_+) = \{\g \in \G: \g x \in \I(\xi_-, \xi_+)\}$ is infinite for $\lambda_-\otimes\lambda_+$-a.e. $(\xi_-, \xi_+)\in \partial X\times \partial X$ whenever $x\in B(h_1,k_1) \times \cdots \times B(h_n,k_n)$.
\end{cor}

\section{Maximality}

\subsection{The Kaimanovich Strip Condition}\label{Kaimanovich Strip}

\begin{definition}
 A pseudo-norm on a group $\G$ with identity $1_\G$ is a map $|\cdot| : \G \to \R_{\geq 0}$ such that for all $\g, \g' \in \G$:
\begin{description}
\item[(normalized)] $|1_\G| = 0$;
\item[(symmetric)] $|\g^{-1}|= |\g|$;
\item[(subadditive)] $|\g\g'|\leq |\g|+|\g'|$.
\end{description}
A pseudonorm satisfying the property that $|\g|\to \8$ as $\g\to \8$  is said to be \emph{proper}. Furthermore, if $|\g| = 0$ implies that $\g=1_\G$ then it is called a norm.
\end{definition}

Suppose that $\G$ acts by isometries on $(X,d)$. Fix a base point $ \o\in X$. This allows us to consider the associated pseudonorm $ |\g|_\o = d(\g \o,  \o)$ on $\G$ which in turn yields the following nested increasing subsets which exhaust $\G$: 
$$\Ga_k =  \Ga_k (\o)= \{ \g \in \G : d(\g \o,  \o) \leq k\}.$$

\begin{definition}
For a fixed pseudo-norm $|\cdot|: \G \to \R$, a probability measure $\mu$ on $\G$  is said to have 
\begin{itemize}
\item \emph{finite first logarithmic moment} (with respect to $|\cdot|$) if 
$$\Sum{\g\in \G}{}\mu(\g)\log|\g|<\8;$$
\item \emph{finite entropy} if $H(\mu):= -\Sum{\g\in \G}{}\mu(\g)\log\mu(\g)<\8$.
\end{itemize}
\end{definition}

We would like to know that under reasonable conditions, the random walk, when translated to an orbit on $X$ is \emph{transient}.\footnote{ Recall that a random walk on a discrete countable group is said to be \emph{transient} if for $\mathbf P'$-a.e. $\omega' \in \Omega'$ we have $\omega_n \to \8$ in $\G$.} 

\begin{lemma}\label{Transient}
Let $\mu$ be a generating probability measure on the nonamenable group $\G$. 
 Then, for any proper pseudonorm $|\cdot|: \G\to \R_+$ and $\mathbf P'$-a.e. $\omega' \in \Omega'$ we have that $|\omega_n|\to \8$. 
\end{lemma}

\begin{proof}
The $\mu$-random walk is transient since $\G$ is nonamenable and $\mu$ is generating \cite[Theorem 2]{DerriennicGuivarc'h}. Since $|\cdot|: \G\to \R_+$ is a proper function, it follows that $|\g|\to \8$ precisely when $\g\to \8$ in $\G$. 
\end{proof}

We will also require the following whose proof is straightforward: 

\begin{lemma}\label{Cto1}
 Let $ \o \in X$ be a point so that the $\stab_\G(\o)$ is finite and set $C= \#\stab_\G(\o)$. If $S \subset X$ then 
 $$\#\{\g \in \G: \g \o \in S\} = C\cdot\#(S\cap \G\cdot \o).$$
\end{lemma}

And finally, we have our main tool for showing maximality:

\begin{theorem}\cite[Strip Condition]{Kaimanovich}\label{Strip Condition}
 Let $\mu$ be a probability measure with finite entropy $H(\mu)$ on a countable group $\G$, and let $(B_-, \lambda_-)$ and $(B_+, \lambda_+)$ be $\ch\mu$- and $\mu$-boundaries, respectively. If there exists a pseudo-norm $|\cdot|:\G \to \R$ and a measurable $\G$-equivariant map $S: B_-\times B_+ \to 2^\G$ such that for all $\g \in \G$ and $\lambda_-\otimes \lambda_+$-a.e. $(b_-,b_+) \in B_-\times B_+$ 
 $$\frac{1}{n} \log \#[S(b_-,b_+)\cdot \g \cap \mathcal{G}_{|\om_n'|}] \underset{n\to \8}{\longrightarrow} 0$$
 in measure $\mathbf{P}'$ on the space of sample paths $\~\om '\in \Om'$, then the boundary $(B_+, \lambda_+)$ is maximal. 
\end{theorem}

\subsection{Proof of Maximality}
By Theorem \ref{BoundaryMap} if $\G\to \Aut(X)$ is a nonelementary and essential action on the finite dimensional CAT(0) cube complex $X$ then there exists probability measures $\lambda_\pm$ on the Roller boundary $\partial X$ so that $(\partial X, \lambda_-)$ and $(\partial X, \lambda_+)$ are $\ch\mu$ and $\mu$-boundaries, respectively.

\begin{theorem}\label{maximal}
Let  $\G$ be a countable discrete group. 
Assume $\G\to \Aut(X)$ is a nonelementary, essential, and proper action on the finite dimensional CAT(0) cube complex $X$. Let $\mu$ be a probability measure  on $\G$ with finite entropy $H(\mu)$. If there is a base point $\o\in X$ for which $\mu$ has  finite first logarithmic-moment $\Sum{\g\in \G}{}\mu(\g)\log|\g|_\o<\8$ then $\partial X$ admits probability measures $\lambda_-$ and $\lambda_+$ making it the Furstenberg-Poisson boundary for $\ch\mu$ and $\mu$, respectively.
\end{theorem}

\begin{proof}

Assume $\G\to \Aut(X)$ is a nonelementary, essential, and proper action on the CAT(0) cube complex $X$. Fix a generating probability measure $\mu$ of finite entropy. By Corollary \ref{PFBoundaryMap} there exists $\lambda_\pm \in \P(\partial X)$ so that $(B_-, \lambda_-): = (\partial X, \lambda_-)$ and $(B_+, \lambda_+) := (\partial X, \lambda_+)$ are $(\G,\ch\mu)$ and $(\G,\mu)$ boundaries, respectively. Fix a base point $\o\in X$ for which $\mu$ has finite first logarithmic moment with respect to the pseudonorm induced by $\o$; it will be denoted by $|\cdot|$.

Also recall that Corollary \ref{cor:StripsNonempty} guarantees the existence of disjoint half-spaces $h_i,k_i$ strongly separated in $\frakH_i$ such that if $\o'\in B(h_1,k_1) \times \cdots \times B(h_n,k_n)$ is a bridge point then for $\lambda_-\otimes\lambda_+$-a.e. $(b_-, b_+)$   these associated strips are infinite:
$$S(b_-, b_+)= \{\g\in \G: \g \o' \in \I(b_-,b_+)\}.$$ 

Maximality of the boundary will follow from Kaimanovich's Strip Condition, Theorem \ref{Strip Condition} if we show that for every $\g_0 \in \G$, and $\lambda_-\otimes\lambda_+$-a.e. $(b_-, b_+)\in B_-\times B_+$ the following converges in measure $\mathbf{P}$ for $\~\om\in \Om$:
$$
 \frac{1}{n}\log\#[S(b_-,b_+)\cdot \g_0\cap \Ga_{|\om_1\cdots\om_{n}|}] \underset{n\to \8}{\longrightarrow}0.
 $$

 To this end, fix $\g_0 \in \G$ and a generic $(b_-, b_+) \in B_-\times B_+$ and note that:
\begin{eqnarray*}
S(b_-,b_+) \g_0 &=& \{\g\in \G: \g\g_0^{-1} \o' \in \I(b_-,b_+)\} \\
&=& \{\g\in \G: \g_0\g\g_0^{-1} \o' \in \I(\g_0 b_-,\g_0 b_+)\}\\
&=& \g_0^{-1}S(\g_0b_-,\g_0b_+)\g_0.
\end{eqnarray*}

Observe that by subadditivity and symmetry $|\g_0\g\g_0^{-1}| \leq |\g| + 2|\g_0|$, and so if $R>0$ then $\g_0 \Ga_R\g_0^{-1} \subset \Ga_{R+2|\g_0|}$. We deduce:
\begin{eqnarray*}
 \#\[S(b_-, b_+)\cdot \g_0\cap \Ga_R\] 
 &\leq &  \#\[S(\g_0b_-, \g_0b_+)\cap\Ga_{R+2|\g_0|
 }\].
\end{eqnarray*}

Next,  we would like to bound $\#[S(b_-,b_+)\cap \Ga_{R}]$ as a function of $R$ that is independent of the generic point $(b_-, b_+)$.

By comparing the volume in the $\ell^1$-metric in $\Z^D$ to the $\ell^\8$-metric we see that the cardinality of any ball of radius $R$ in Euclidean $D$-space is bounded above by $(2R+1)^D$.  Theorem \ref{th:Euclidean equivalences} guarantees the existence of an isometric embedding $\I(b_-,b_+) \hookrightarrow \~\R^D$ and so, it follows that a ball of radius $R$ in an interval also has cardinality bounded above by   $(2R+1)^D$. 

 Fix $x\in \I(b_-, b_+)\cap N_R(\o')$, where $N_R(\o')$ denotes the ball of radius $R$ in $X$ centered at $\o'$. Then applying the triangle inequality we deduce that   
 $$\I(b_-, b_+)\cap N_R(\o') \subset \I(b_-, b_+)\cap N_{2R}(x) $$
  and by the previous paragraph
$$\#\(\I(b_-, b_+)\cap N_R(\o') \) \leq \#\( \I(b_-, b_+)\cap N_{2R}(x) \) \leq (4R+1)^D.$$

We also have that $d(\g\o',\o') \leq 2d(\o,\o') + d(\g\o,\o)$ and so setting $\delta= 2d(\o,\o')$ and $C'= \#\stab_\G(\o')$ we apply Lemma \ref{Cto1} to deduce:
\begin{eqnarray*}
\#[S(b_-,b_+)\cap \Ga_{R}] &\leq &\# \{\g: d(\g\o',\o')\leq R+2d(\o,\o') \text{ and } \g\o'\in \I(b_-,b_+)\}\\
&=& C' \cdot \#\( \I(b_-,b_+)\cap N_{R+\delta}(\o')\cap \G\cdot \o'\)\\
&\leq & C'\cdot [4(R+\delta) + 1]^D.
\end{eqnarray*}

By Lemma \ref{Transient} and Corollary \ref{cor:StripsNonempty} we have that for $\mathbf{P}$-a.e. $\~\omega \in \Omega$, $\lambda_-\otimes\lambda_+$-a.e. $(b_-,b_+)$, and  $n$ sufficiently large the following is nonempty:  
$$S(b_-, b_+)\cdot \g_0\cap \Ga_{|\om_1\cdots\om_n|}.$$ 
And so:
\begin{eqnarray*}
0&\leq&\frac{1}{n} \log \#\(S(b_-, b_+)\cdot \g_0\cap \Ga_{|\om_1\cdots\om_n|}\) 
\\%
&\leq& \frac{1}{n} \log \#\(S(\g_0b_-,\g_0 b_+)\cap \Ga_{|\om_1\cdots\om_n| +2|\g_0|}\)
\\ %
&\leq &
\frac{1}{n} \log \(C'\cdot [4\( |\om_1\cdots\om_n| +2|\g_0|+\delta \) + 1]^D\).
\end{eqnarray*}
Therefore, the quantity on the first line converges in measure to 0 if the quantity on the last line converges  to 0 in measure, if and only if the following converges  in measure:
$$ \frac{1}{n} \log  |\om_1\cdots\om_n|\to 0. $$
To this end, observe that since the pseudonorm is subadditive, we have that  
$$ \frac{1}{n} \log  |\om_1\cdots\om_n| \leq  \frac{1}{n} \log \(\,\overset{n}{\Sum{k = 1}} |\om_k|\) .$$
The right hand side of this inequality converges in measure to 0 since $\mu$ has finite first logarithmic moment with respect to $|\cdot|$ (see Proposition 2.3.1 \cite{Aaronson}).

\end{proof}

\section{Proof of the Tits' Alternative}

We have gathered almost all the necessary tools for the proof of the Tits' Alternative. Here are a few more.

\begin{prop}\label{Roller-Elementary means invariant interval}
If $\G\to \Aut(X)$ is Roller-elementary, then there is $v,w\in \~X$ such that $\G\cdot \I(v,w) = \I(v,w)$. Furthermore, if $\G$ has an orbit whose cardinality is an odd integer, then $\G$ has a fixed point in $\~X$.
\end{prop}

\begin{proof}
 Let $ o\subset \~X$ be a point whose $\G$-orbit is  finite. Let $m$ be the average of the Dirac masses on $\G\cdot o$. Then clearly $\G$ preserves the measure $m$ and therefore $\~X_m$. Furthermore, if $ o$ is an orbit whose cardinality is an odd integer, then $m(h) \neq 1/2$ for every $h\in \frakH$, and in particular, $H_m^+$ satisfies both consistency and totality meaning that $\~X_m$ is a single point. 
 \end{proof}
 
 From this we obtain the following version of a classical result of Adams and Ballman which states that for a locally compact Hadamard space if $\G$ is an amenable group acting by isometries then it either fixes a point in the visual boundary or preserves a flat \cite{AdamsBallmann}. This result has been generalized in many contexts such as \cite{CapraceLytchak}, \cite{CapraceMonod}.
 
\begin{lemma}\label{AdamsBallmann}
Any action of an amenable group on a CAT(0) cube complex is Roller-elementary. Furthermore, if there is an odd orbit then there is a fixed point in $\~X$, and otherwise there is an invariant interval. 
\end{lemma}

\begin{proof}
 Suppose $\G$ is an amenable group acting on $X$. Then, it admits an invariant probability measure $m$ on $\~X$. By invariance, we have that $H_m$ and $H_m^+$ are invariant as well, and hence $\~X_m$ is $\G$-invariant. By Corollary \ref{Interval on Finitely Many}, there are finitely many elements on which $\~X_m$ is an interval and therefore, that set is $\G$-invariant, hence the $\G$-action is Roller-elementary.
 \end{proof}

Recall that a subgroup is said to be $X$-locally elliptic if every finitely generated subgroup has a fixed point. Also, a subgroup is said to be the $X$-locally elliptic radical if it is the unique maximal normal subgroup which is $X$-locally elliptic. We then have the following (see also \cite[Theorem A.5]{CapraceLecureux}):

\begin{theorem}\cite[Caprace's Theorem B1]{CFI}\label{CapraceShortExact} 
Consider the stabilizer $\stab(x) \leq \Aut(X)$, for $x\in \partial X$. Then there is a virtually abelian group   $A$ of rank $n\leq \mathrm{dim}(X)$ such that if $N$ is the $X$-locally elliptic radical then we have an exact sequence:
$$1\to N\to \stab(x) \to A \to 1.$$
\end{theorem}

We will now turn to the proof of the following:

\begin{theoremnn}[Tits' Alternative]
 Let $X$ be a  finite dimensional CAT(0) cube complex and $\G \leq \Aut(X)$.  The following are equivalent:
\begin{enumerate}
\item $\G$ does not preserve any interval $\I \subset \~X$.
\item The $\G$-action is  Roller nonelementary. 
\item $\G$ contains a  nonabelian  free subgroup acting freely on $X$. 
\end{enumerate}
Furthermore, these imply:
\begin{enumerate}
\item[(4)] The closure $\~\G$ in $\Aut(X)$ is nonamenable;
\end{enumerate}
and if $X$ is locally compact then then (4) implies (1).
\end{theoremnn}

\begin{proof}
 
 \medskip
(\ref{Interval}) $\implies $ (\ref{nonelem}): This is Proposition \ref{Roller-Elementary means invariant interval}.

 \medskip
\noindent
(\ref{nonelem}) $\implies$ (\ref{free subgroup}): 
If the action is Roller nonelementary then by Proposition \ref{prop:visual-to-roller} there is a finite index subgroup $\G'\leq \G$ and a subcomplex $X'\subset \~X$ with half-space structure $\frakH'$ where the $\G'$-action is essential and nonelementary. By Corollary \ref{finallyIntervalEuclidean}, there exists  facing triple $h_1, h_2, h_3\in \frakH'$ so that as subsets of $X'$ they are facing. 
Applying the  Double Skewering Lemma \ref{Double Skewering Lemma}, we may assume there is a facing quadruple $h_1, h_2,h_3, h_4 \in \frakH'\subset \frakH$. Once more, the Double Skewering Lemma \ref{Double Skewering Lemma} guarantees  the existence of $a,b\in \G'$ such that $a h_1\subset h_2^* \subset h_1$, and $b h_3\subset h_4^* \subset h_3$. Letting $A= h_1$, and $B= h_3$, we see that (the complements of) $A,aA^*,B, bB^*$ form a standard Ping-Pong table and so the obvious  homomorphism  $F_2\to \<a,b\>$ is an isomorphism. We note that by Remark \ref{rem half-spaces in proj}, these half-spaces as subsets of $X$ have the same relationships and hence we now think of them as subsets of $X$.

Without loss of generality, we may assume that the $F_2$ orbits of $A, aA^*, B, bB^*$ are pairwise disjoint. Indeed, we may pass to the cubical subdivision if necessary to assure that  $F_2\cdot\{A\}\cap F_2\cdot\{aA^*\}=\varnothing = F_2\cdot\{B\}\cap F_2\cdot\{bB^*\}$. And up to replacing this copy of $F_2$ with another (i.e. $\<a^p,b^q\>$ for some $p,q\in \N$, we may assume that $F_2\cdot\{A\}\cap F_2\cdot\{bB^*\}=\varnothing= F_2\cdot\{B\}\cap F_2\cdot\{A\}$.
 
 We now show that this  action is free. To this extent, observe that the proof of the Ping-Pong Lemma shows that if there is a  point fixed by an element of $F_2\setminus\{1\}$, then it does not belong to $\mathcal F := A\cap aA^*\cap B\cap bB^*$ or any of its translates. Therefore, we show that if $x\in X$ then $x\in w\mathcal F$ for some $w\in F_2$. 
Let us make an easy observation: if $o \in w\mathcal F= wA\cap wB\cap wbB^* \cap waA^*$ with $\#\([o, x]\cap F_2\cdot\{A,aA^*,B,bB^*\}\) =0$ then $x\in w\mathcal F$. Therefore, we aim to produce such a pair $w\in F_2$ and $o \in w\mathcal F$. 

Fix $\o \in \mathcal F$. Consider the linearly ordered (by inclusion) set of half-spaces 
$$[\o,x]\cap F_2\cdot\{A,aA^*,B,bB^*\} =: \{w_1h_1, \cdots w_nh_n\}$$
 where $w_i \in F_2$ and $h_n \in \{A,aA^*,B,bB^*\}$, with $w_i h_i\subset w_{i+1}h_{i+1}$. Let $o = m(\o, w_n \o, x)$. We claim that $o \in w_n \mathcal F$. Indeed, $w_n h_n \in U_x\cap U_{w_n \o}$ and so $w_nh_n\in U_o$. 
 
Observe that $w_n\cdot \{A,B,aA^*, bB^*\}\setminus[o,x] = w_n\cdot \{A,B,aA^*, bB^*\}\setminus \{w_nh_n\}$. Indeed, the half-spaces $w_n\cdot \{A,B,aA^*, bB^*\}$ are facing, and $w_n h_n$ is minimal in $[o,x]\cap F_2\cdot \{A,B,aA^*, bB^*\} $. Therefore, if $h\in w_n\cdot \{A,B,aA^*, bB^*\}\setminus \{w_nh_n\}$ then either both $\o, x\in h$ or both $\o, x\in h^*$. Again, using that these half-spaces are facing, if $\o, x\in h^*$ then $\o \in h^*\cap w_nh_n^*= \varnothing$. Therefore, $\o, x\in h$ and $h\in U_o$. 

Observe $o \in w_n\mathcal F= wA\cap wB\cap wbB^* \cap waA^*$. By construction we have that 
$$[o,x]\cap F_2\cdot \{A,B,aA^*, bB^*\}\subset \([\o,x]\cap F_2\cdot \{A,B,aA^*, bB^*\}\)\setminus \{w_1h_1, \dots, w_nh_n\} = \varnothing$$ 
and hence $x\in w_n\mathcal F$.

 \medskip
\noindent
(\ref{free subgroup}) $\implies$ (\ref{Interval}): Suppose that $\G$ preserves the interval $\I$. If $\I \subset X$ then it is finite and $\G$ has a finite index subgroup fixing each element of $\I$ therefore, no free subgroup can act freely.

Suppose now that $\I = \I(x,y)$ with $x\in \partial X$. Once more, $\G$ has a finite index subgroup $\G_0$ by Corollary \ref{Interval on Finitely Many}, fixing $x$. Let $F\leq \G$ be a nonabelian free subgroup and  $F_n = \G_0 \cap F$. Since $F_n$  is finitely generated, by Caprace's Theorem \ref{CapraceShortExact}, the commutator subgroup $[F_n,F_n]$ has a fixed point in $X$, and hence the action of $F$ is not free.

 \medskip
\noindent
(\ref{free subgroup}) $\implies$ (\ref{nonamen closure}): This follows from the basic fact that containing a closed free subgroup is an obstruction to amenability.

 \medskip

\noindent
(\ref{nonamen closure}) $\implies$ (\ref{Interval}): Assume that $X$ is locally compact.  By 
Corollary \ref{Interval on Finitely Many}, preserving an interval implies the existence of a finite index subgroup with a fixed point in $\~X$. Therefore, the contrapositive will be shown if we demonstrate that point stabilizers are amenable. If $x\in X$ the stabilizer $\stab(x)$ is compact and hence amenable. It follows that if $x\in \partial X$ then  $N_x$ the $X$-locally elliptical radical of $\stab(x)$ is also amenable. Indeed, $N_x$ is a union of the compact groups $\stab(y) \cap \stab(x)$, where $y\in X$ and is hence amenable. Applying  Caprace's Theorem \ref{CapraceShortExact} once more, finishes the proof. 
 \end{proof}

\section{Convergence Actions}\label{Convergence Section}

An action of a countable discrete group $\G$ by homeomorphisms on a compact metrizable space $M$ is said to be a \emph{convergence action} if the diagonal action on the space of distinct triples is proper. Namely, if $x,y,z\in M$ are pairwise distinct, and $\g_n\in \G$ any sequence, then up to passing to a subsequence $\#\{\g_nx, \g_n y, \g_n z\} \to N<3$. 

Bader and Furman show that if $\G$ admits a convergence action on $M$ then there is a $\G$-equivariant map $\phi : B \to M$, where $B$ is a  Furstenberg-Poisson boundary of  $\G$ \cite[Theorem 3.2]{BaderFurman}. In this section, we show that the action of $\G$ on the Roller compactification $\~X$ is not a convergence action if there is an interval $\I\subset \~X$ with the following properties:

\begin{itemize}
\item The stabilizer of $\I$ in $\G$ is infinite.
\item There exist $x_1,x_2,y_1, y_2\in \I$ such that $\I = \I(x_1, y_1)=\I(x_2, y_2) $ and 
$$\#\{x_1,x_2,y_1, y_2\} \geq 3.$$
\end{itemize}

We observe that these are rather weak conditions on higher dimensional CAT(0) cube complexes, for example they are satisfied if $\G$ has commuting independent hyperbolic elements. 

Let us now show that under the above conditions, the diagonal action of $\G$ on distinct triples is not proper. To this end, recall that since $\I$ is an interval, it embeds into $\~\Z^D$, where $D$ is the dimension of $X$. Then, the collection of $x\in \I$ for which there is an $y\in \I$ such that $\I = \I(x,y)$ has cardinality bounded above by $2^D$  by Corollary \ref{Interval on Finitely Many} and  bounded below by 3,  by assumption. Therefore, there is a subgroup of finite index which fixes each of these elements. The assumption that the stabilizer of $\I$ is infinite implies that the point-wise fixator of each of these points is also infinite. Therefore, any distinct triple from that set  has an infinite stabilizer and hence the action is not a convergence action. 

\bibliographystyle{alpha}

\bibliography{CCCPoisson}

\end{document}